\documentclass[11pt,
]{article}

\usepackage{amsmath}
\usepackage{amsfonts}
\usepackage{amsthm}
\usepackage{amssymb}
\usepackage{latexsym}
\usepackage{fixmath}
\usepackage{mathdots}

\usepackage{kbordermatrix}

\usepackage[]{cleveref}
\crefname{assumption}{assumption}{assumptions}
\crefname{figure}{figure}{figures}
\crefname{equation}{}{}
\crefname{subsection}{subsection}{subsections}

\usepackage[shortlabels]{enumitem}
\setlist[enumerate,2]{label=(\alph*),ref=\theenumi(\alph*)}
\setlist[enumerate,3]{label=\roman*.,ref=\theenumii\roman*}

\usepackage{mathtools}

\DeclarePairedDelimiterXPP\set[1]{}{\{}{\}}{}{
	
#1}
\DeclarePairedDelimiter\abs{|}{|}
\def\wtd{\widetilde}
\def\what{\widehat}
\def\sss{\scriptscriptstyle}
\def\ii{\mathrm{i}}
\def\ee{\mathrm{e}}

\def\bbI{\mathbb{I}}
\def\bbR{\mathbb{R}}
\def\bbC{\mathbb{C}}
\def\cR{{\cal R}}

\DeclareMathOperator{\diag}{diag}
\DeclareMathOperator{\eig}{eig}
\DeclareMathOperator{\opt}{opt}
\DeclareMathOperator{\rank}{rank}
\DeclareMathOperator{\sign}{sign}
\DeclareMathOperator{\trace}{tr}
\DeclareMathOperator{\HH}{H}
\DeclareMathOperator{\T}{T}

\newtheorem{theorem}{Theorem}[section]
\newtheorem{lemma}{Lemma}[section]

\theoremstyle{definition}
\newtheorem{definition}{Definition}[section]
\newtheorem{remark}{Remark}[section]

\numberwithin{equation}{section}
\numberwithin{figure}{section}
\numberwithin{table}{section}
\allowdisplaybreaks

\def\inertia{\mathfrak{i}}
\def\properind{\mathfrak{d}}
\newcommand\block[2]{\textbf{T-#1(#2)}}
\def\blocko{\textbf{T-o}}
\usepackage{mathrsfs}
\newcommand\group[1]{\mathcal{ #1}}
\usepackage{amsbsy}
\newcommand\bs[1]{\boldsymbol{ #1}}
\def\one{\bs{1}}
\def\be{\bs{e}}
\def\bu{\bs{u}}
\def\bx{\bs{x}}

\title{On Generalizing Trace Minimization Principles, II}

\author{
	Xin Liang%
	\thanks{
		Yau Mathematical Sciences Center, Tsinghua University, Beijing 100084, China, and Yanqi
		Lake Beijing Institute of Mathematical Sciences and Applications, Beijing 101408, China.
		E-mail: {\tt liangxinslm@tsinghua.edu.cn}.  
	}
	\and
	Ren-Cang Li%
	\thanks{Department of Mathematics,
		University of Texas at Arlington, Arlington, TX 76019-0408, USA. Supported in part by NSF DMS-2009689.
	Email: {\tt rcli@uta.edu}.}
}
\date{\today}

\begin{document}
\maketitle
\begin{abstract}
	This paper is concerned with establishing a trace minimization principle for two Hermitian matrix pairs. Specifically, we
	will answer the question: when is \linebreak $\inf_X\trace(\what AX^{\HH}AX)$ subject to $\what BX^{\HH}BX=I$
	(the identity matrix of apt size) finite? Sufficient and necessary conditions are obtained and, when the infimum is finite, an explicit formula for it  is presented in terms of the finite eigenvalues of the matrix pairs. Our results
	extend  Fan's trace minimization principle (1949) for a Hermitian matrix, a minimization principle
	of Kova{\v c}-Striko and K.~Veseli{\' c} (1995) for a Hermitian matrix pair, and most recent ones by the authors and their collaborators for a Hermitian matrix pair and a Hermitian matrix.
\end{abstract}



\section{Introduction}\label{sec:introduction}
Various trace minimization principles have served as the theoretical foundations for computing eigenvalues
of special kinds of matrix pairs  and played important roles in numerical linear algebra
\cite{bali:2012a,bali:2013,bali:2014,ball:2016,knya:2001,krps:2013,li:2004c,li:2015,lili:2015,sawi:1982,soag:2021}.
Fan's trace minimization principle \cite{fan:1949} \cite[p.248]{hojo:2013} is perhaps the earliest and the most well-known one:
\[
	\min_{X^{\HH}X=I_k}\trace(X^{\HH}AX) = \sum_{i=1}^{k} \lambda_i,
\]
where $A\in \bbC^{n\times n}$ is Hermitian with its eigenvalues denoted by $\lambda_1\le\lambda_2\le \dots\le \lambda_n$,
$\trace(\,\cdot\,)$ takes the trace of a matrix, $I_k$ is the $k\times k$ identity matrix.
Moreover, any minimizer $X_{\min}$ is an orthonormal basis matrix of the invariant subspace of $A$ associated with its eigenvalues $\lambda_1,\dots,\lambda_k$. It has since been
generalized to many broader cases:
\begin{enumerate}
	\item The most straightforward generalization is
		$\min\limits_{X^{\HH}BX=I_k}\trace(X^{\HH}AX)$ for a Hermitian matrix pair $(A,B)$, where
		$B$ is positive definite.
	\item For a Hermitian matrix pair $(A,B)$ with indefinite and possibly singular $B$, in
		\cite{kove:1995,lilb:2013,liangL2014extensions}
		$\inf\limits_{X^{\HH}BX=J_k}\trace(X^{\HH}AX)$ is investigated, where $J_k$ is a $k\times k$
		diagonal matrix with diagonal
		entries $\pm1$. It is shown that the infimum is finite if and only if
		$(A,B)$ is a positive semidefinite matrix pair, by which we mean there exists $\lambda_0\in\bbR$ such that $A-\lambda_0B$ is positive semidefinite.
	\item From the perspective of optimization, in
		\cite{liuSW2019quadratic}
		$\min\limits_{X^{\HH}X=I_k}\trace(\what AX^{\HH}AX)$ is  analyzed, where both $A,\,\what A$ are Hermitian
		matrices.
	\item More recently,  the authors of \cite{liwz:2023} investigated two more general cases:
		\begin{enumerate}
			\item $\min\limits_{X^{\HH}BX=I_k}\trace(\what AX^{\HH}AX)$
				where $A,\,B,\,\what A$ are Hermitian matrices and $B$ is positive definite;
			\item  $\inf\limits_{X^{\HH}BX=J_k}\trace(\what AX^{\HH}AX)$ where $J_k=\begin{bmatrix}
					I_{k_+} & \\ & -I_{k_-}
					\end{bmatrix}$ and $\what A=\begin{bmatrix}
					\what A_+ & \\ & \what A_-
				\end{bmatrix}$ have the same block-diagonal structure,  $\what A_\pm$ are of size $k_{\pm}$,
				$A,\,B,\,\what A$ are Hermitian matrices,
				and $(A,B)$ is a positive semidefinite matrix pair.
		\end{enumerate}
\end{enumerate}
Our goal in this paper, as a continuation of \cite{liwz:2023}, is to investigate yet an even more general case:
\begin{equation*}\label{eq:mainP}
	\inf\limits_{\what BX^{\HH}BX=I_k}\trace(\what AX^{\HH}AX),
\end{equation*}
where $A,B,\what A,\what B$ are Hermitian of apt sizes. Our main result relates this infimum to two matrix pairs $(A,B)$ and $(\what A,\what B)$.
With $\what B=I_k$, it reduces to 4(a) and with $\what B=J_k$ it reduces to 4(b) above.

The rest of this paper is organized as follows.
We review the basics about positive semidefinite matrix pairs in \cref{sec:prelim}, and
state our main result of this in paper in \cref{sec:main-result}.
The proof of the main result spreads in the next two sections: \cref{sec:the-simple-case}
deals with the simple case when $n=\what n$ and both pairs $(A,B)$ and $(\what A,\what B)$ are congruent-diagonalizable
while \cref{sec:the-general-case} handles the main result in its generality, with the help of the
result in the simple case. We draw our concluding remarks in \cref{sec:conclusion}.

\paragraph{Notations.}
Throughout this paper, ${\mathbb C}^{n\times m}$ is the set
of  $n\times m$ complex matrices, ${\mathbb C}^n={\mathbb C}^{n\times 1}$,
and ${\mathbb C}={\mathbb C}^1$, and their real counterparts are denoted similarly by replacing $\bbC$ with $\bbR$.
By $\group U_n,\group P_n,\group D_n\in\bbC^{n\times n}$, denote the sets of unitary, permutation, diagonal matrices, respectively (and by $\group U,\group P,\group D$ if their sizes are clear from the context);
and, by $\group P^u_n,\group D^+_n \in\bbC^{n\times n}$, denote the set
of permutation matrices in structure but with nonzero entries being any unit complex number $\ee^{\ii\theta}$ and that of diagonal matrices with nonnegative diagonal entries, respectively (and by $\group P^u,\group D^+$ if their sizes are clear from the context).
$I_n$ (or simply $I$ if its dimension is clear from the context) is the $n\times n$ identity matrix.

For a matrix $X\in\bbC^{m\times n}$, ${\cal N}(X)=\set{\bx\in\bbC^n\,:\,X\bx=0}$ and $\cR(X)=\set{X\bx\,:\,\bx\in\bbC^n}$
are the null space and the range of $X$ (also known as the column space of $X$), respectively.
$X^{\T}$ and $X^{\HH}$ are the transpose and the conjugate transpose of a vector or matrix, respectively.
$A\succ 0$ ($A\succeq 0$) means that $A$ is Hermitian positive (semi)definite, and $A\prec 0$ ($A\preceq 0$) if
$-A\succ 0$ ($-A\succeq 0$).
The eigenvalues of an $n\times n$ Hermitian matrix $A$ are written, according to either increasing or decreasing order,
as
\[
	\lambda_1^{\uparrow}(A)\le\lambda_2^{\uparrow}(A)\le\dots\le\lambda_n^{\uparrow}(A),
	\quad\text{or}\,\,
	\lambda_1^{\downarrow}(A)\ge\lambda_2^{\downarrow}(A)\ge\dots\ge\lambda_n^{\downarrow}(A),
\]
respectively. Hence $\lambda_i^{\uparrow}(A)=\lambda_{n-i+1}^{\downarrow}(A)$.

Other notational convention will be introduced as they appear for the first time.

\section{Preliminaries on a positive semidefinite matrix pair}
\label{sec:prelim}
We review some of related concepts and
results about a positive semidefinite matrix pair $(A,B)$ \cite{lilb:2013}.

Given Hermitian $B\in\bbC^{n\times n}$, the {\em inertia\/} of $B$ is the integer
triplet $(\inertia_+(B), \inertia_0(B), \inertia_-(B))$, meaning $B$ has $\inertia_+(B)$ positive, $\inertia_0(B)$ zero, and $\inertia_-(B)$ negative eigenvalues,
respectively. Necessarily
\begin{equation*}\label{eq:rankB}
	r:= \rank(B)=\inertia_+(B)+\inertia_-(B).
\end{equation*}

Consider $n\times n$ matrix pair $(A,B)$. We say $\mu\ne\infty$ is a {\em finite eigenvalue\/}
of $(A,B)$ if
\begin{equation*}\label{eq:eig-dfn}
	\rank(A-\mu B)<\max_{\lambda\in{\mathbb C}}\rank(A-\lambda B),
\end{equation*}
and $\bx\in{\mathbb C}^n$ is a corresponding {\em eigenvector\/} if $\bx\not\in{\cal N}(A)\cap{\cal N}(B)$ satisfies
\begin{equation*}\label{eq:eigv-dfn}
	A\bx=\mu B\bx,
\end{equation*}
or equivalently, $\bx\in{\cal N}(A-\mu B)\backslash({\cal N}(A)\cap{\cal N}(B))$. Together $(\mu,\bx)$ is
called an {\em eigenpair\/} of $(A,B)$.

\begin{definition}[\cite{lilb:2013,kove:1995}]\label{dfn:HP-SDP}
	$(A, B)$ is a {\em Hermitian matrix pair\/} of order $n$ if both $A,\,B\in{\mathbb C}^{n\times n}$ are
	Hermitian. $(A, B)$ is a {\em positive (semi)definite matrix pair\/} of order $n$ if it is a Hermitian matrix pair
	of order $n$ and if there exists $\lambda_0\in{\mathbb R}$ such that $A-\lambda_0B$ is positive (semi)definite, in notation, $(A,B)\succ 0$ ($\succeq 0$). $(A, B)\prec 0$ ($\preceq 0$) if $(-A,-B)\succ 0$ ($\succeq 0$).
\end{definition}

Given an eigenpair $(\mu,\bx)$ of  a Hermitian matrix pair $(A,B)$, we say $\mu$ is an eigenvalue of positive type if
$\bx^{\HH}B\bx>0$ and of negative type if $\bx^{\HH}B\bx<0$.

Let $(A,B)$ be a positive-semidefinite matrix pair of order $n$ as in the definition.
It is known \cite[Lemma 3.8]{lilb:2013} that
$(A, B)$ has only $r=\rank(B)$ finite eigenvalues all of which are real.
Denote these finite eigenvalues, according to either increasing or decreasing order, by
\[
	\lambda_1^{-\uparrow}(A,B)\le\dots\le\lambda_{\inertia_-(B)}^{-\uparrow}(A,B)
	\le\lambda_1^{+\uparrow}(A,B)\le\dots\le\lambda_{\inertia_+(B)}^{+\uparrow}(A,B),
\]
or
\[
	\lambda_1^{+\downarrow}(A,B)\ge\dots\ge\lambda_{\inertia_+(B)}^{+\downarrow}(A,B)\ge\lambda_1^{-\downarrow}(A,B)\ge\dots\ge\lambda_{\inertia_-(B)}^{-\downarrow}(A,B).
\]
Since both $\set{\lambda_i^{-\uparrow}(A,B),\lambda_j^{+\uparrow}(A,B)}_{i,j}$ and
$\set{\lambda_i^{-\downarrow}(A,B),\lambda_j^{+\downarrow}(A,B)}_{i,j}$
are the same set of finite eigenvalues of $(A,B)$, we will have
\[
	\lambda_i^{-\uparrow}(A,B)=\lambda_{\inertia_-(B)-i+1}^{-\downarrow}(A,B), \quad
	\lambda_j^{+\downarrow}(A,B)=\lambda_{\inertia_+(B)-j+1}^{+\uparrow}(A,B)
\]
for $1\le i\le \inertia_-(B)$
and $1\le j\le \inertia_+(B)$.
Eigenvalues
$\lambda^{+\uparrow}_j(A,B)$ (and $\lambda^{+\downarrow}_j(A,B)$ too)
are those  of positive type (accordingly to their associated eigenvectors $\bx$ that make $\bx^{\HH}B\bx>0$),
whereas $\lambda^{-\uparrow}_j(A,B)$ (and $\lambda^{-\downarrow}_j(A,B)$ too)
are those of negative type (accordingly to  their associated eigenvectors $\bx$ that make $\bx^{\HH}B\bx<0$).
It has also been proved \cite{lilb:2013} that for all $i,\,j$,
\begin{equation*}\label{eq:finite-eig-property}
	\lambda_i^{-\uparrow}(A,B)\le\lambda_0\le\lambda_j^{+\uparrow}(A,B), \quad
	\lambda_j^{+\downarrow}(A,B)\ge\lambda_0\ge\lambda_i^{-\downarrow}(A,B).
\end{equation*}
As a consequence, eigenvalues of positive type are no smaller than those of negative type, i.e.,
\[
	\lambda_j^{+\uparrow}(A,B)-\lambda_i^{-\uparrow}(A,B)\ge 0, \quad
	\lambda_j^{+\downarrow}(A,B)-\lambda_i^{-\downarrow}(A,B)\ge 0.
\]

There is an important comment that needs to be made about the types of the eigenvalues of 
matrix pair $(A,B)\succeq 0$. When $\lambda_0$ in the definition is an eigenvalue, there is
a possibility that $(A,B)$ may have $2$-by-$2$ Jordan block pairs associated 
with eigenvalues $\lambda_0$ (see \Cref{rk:LambdaJ-PD} later):
\begin{equation}\label{eq:lambda2-2x2}
	\left(\begin{bmatrix}
			0 & \lambda_0 \\
			\lambda_0 & 1
			\end{bmatrix},\begin{bmatrix}
			0 & 1 \\
			1 & 0
		\end{bmatrix}
	\right),
\end{equation}
which corresponds to one eigenvector $\bx$ with $\bx^{\HH}B\bx=0$. Each of such Jordan block pairs brings two copies of
$\lambda_0$ as eigenvalues.
In \cite{lilb:2013,liangL2014extensions},
artificially, one copy is regarded as of positive type while the other as of negative type. Although seemingly artificial,
it can be justified by perturbing the first block in the pair to
$\begin{bmatrix}
	\varepsilon & \lambda_0 \\
	\lambda_0 & 1
\end{bmatrix}
$ for $\varepsilon>0$ and letting $\varepsilon\to 0^+$. The perturbation breaks the two copies of
$\lambda_0$ into $\lambda_0+\sqrt{\varepsilon}$ of positive type and $\lambda_0-\sqrt{\varepsilon}$ of negative type.
Any other eigenvalues different from $\lambda_0$ are all associated with Jordan block pairs of $1$-by-$1$.
It can be seen that if $(A,B)\succeq 0$ does have a $2$-by-$2$ Jordan block pair \cref{eq:lambda2-2x2},
then 
$\lambda_1^{-\downarrow}(A,B)=\lambda_1^{+\uparrow}(A,B)=\lambda_0$.
In view of these discussion, we conclude that
\begin{equation}\label{eq:PSD-intvl}
	A-\lambda_0 B\succeq 0\quad\text{for any $\lambda_0\in [\lambda_1^{-\downarrow}(A,B),\lambda_1^{+\uparrow}(A,B)]$}.
\end{equation}
In fact, if $\lambda_1^{-\downarrow}(A,B)<\lambda_1^{+\uparrow}(A,B)$, then 
$(A,B)$ can only have $1$-by-$1$ Jordan block pairs.

Similar statements can be made about the eigenvalues of a negative semidefinite matrix pair.

\section{Main result}\label{sec:main-result}
Once again, we are interested in a minimization principle for

\begin{equation}\label{eq:problem}
	\inf_{\what BX^{\HH}BX=I_{\what n}}\trace(\what A X^{\HH}AX),
\end{equation}
where $A,\,B\in\bbC^{n\times n}$ and $\what A,\,\what B\in\bbC^{\what n\times\what n}$ are all Hermitian matrices,
and $\what n\le n$.
Here we adopt a different notation in $\what n$ from $k$ used in \cref{sec:introduction} to align with
our overall notation structure.
Constraint $\what BX^{\HH}BX=I_{\what n}$ necessarily implies both $\what B$ and $X^{\HH}BX$ are nonsingular and also
\begin{equation}\label{eq:inertia(BtB)}
	\inertia_+(\what B)=\inertia_+(\what B^{-1})=\inertia_+(X^{\HH}BX)\le \inertia_+(B), \quad
	\inertia_-(\what B)\le \inertia_-(B).
\end{equation}


Before stating our main result on \cref{eq:problem}, we  introduce a new notion on Hermitian matrix triplet
$(B,\what A,\what B)$, which we need to express our conditions for the infimum to be finite.

\begin{definition}\label{dfn:proper}
	Given a Hermitian positive semidefinite pair $(\what A,\what B)$ and a Hermitian matrix $B$,
	the triplet $(B,\what A,\what B)$ is said \emph{proper} if one of the following statements holds, where the \emph{proper index pair} $(\properind_+(\what B),\properind_-(\what B))$ is defined along the way:
	\begin{enumerate}[(i)]
		\item $\inertia_+(B)=\inertia_+(\what B)$ and $\inertia_-(B)=\inertia_-(\what B)$, for which
			$(\properind_+(\what B),\properind_-(\what B))=(0,0)$;
		\item $\inertia_+(B)=\inertia_+(\what B),\, \inertia_-(B)>\inertia_-(\what B)$
			and $\lambda_1^{+\uparrow}(\what A,\what B)\ge 0$, for which $\properind_+(\what B)=0$ and
			$\properind_-(\what B)$ is the number of positive ones among $\lambda_j^-(\what A,\what B)$,
			$1\le j\le \inertia_-(B)$;
		\item $\inertia_+(B)>\inertia_+(\what B),\, \inertia_-(B)=\inertia_-(\what B)$
			and $\lambda_1^{-\downarrow}(\what A,\what B)\le 0$, for which $\properind_-(\what B)=0$
			and $\properind_+(\what B)$ is the number of negative ones in $\lambda_j^+(\what A,\what B)$,
			$1\le j\le \inertia_+(B)$;
		\item $\inertia_+(B)>\inertia_+(\what B),\, \inertia_-(B)>\inertia_-(\what B)$
			and $\lambda_1^{-\downarrow}(\what A,\what B)\le 0\le \lambda_1^{+\uparrow}(\what A,\what B)$, for which \linebreak
			$(\properind_+(\what B),\properind_-(\what B))=(0,0)$.
	\end{enumerate}
	Here the dependency of $(\properind_+(\what B),\properind_-(\what B))$ on $B$ and $\what A$ is suppressed for clarity.
	The triplet $(B,\what A,\what B)$ is said \emph{improper} if it is not proper.
\end{definition}

As a corollary of our discussions at the end of \cref{sec:prelim}, the condition
$\lambda_1^{-\downarrow}(\what A,\what B)\le 0\le \lambda_1^{+\uparrow}(\what A,\what B)$ in the case (iv) in the definition
is the same as $\what A\succeq 0$. 

\begin{theorem}\label{thm:J-unitary:total}
	Given four Hermitian matrices $A,B\in \bbC^{n\times n},\what A,\what B\in \bbC^{\what n\times \what n}$
	where $n\ge \what n$,
	suppose that
	$\what A\ne 0$, $A\ne \mu B$ for any $\mu\in \bbR$, and $\what A\ne\what \mu\what B$ for any $\what\mu\in \bbR$
	when $n=\what n$.
	Then
	\[
		\inf\limits_{\what BX^{\HH}BX=I_{\what n}}\trace(\what A X^{\HH}AX)>-\infty,
	\]
	i.e., finite, if and only if one of the following two cases occurs:
	\begin{enumerate}[(i)]
		\item both $(A, B)$ and $(\what A, \what B)$ are positive  semidefinite pairs and
			$(B,\what A,\what B)$ is proper;
		\item both $(A, B)$ and $(\what A, \what B)$ are negative semidefinite pairs and
			$(-B,-\what A,-\what B)$ is proper.
	\end{enumerate}
	Moreover, in the first case,
	we have\footnote {We adopt the convention $\sum_{i=1}^0(\,\cdot\,)\equiv 0$.}
	\begin{align}
		\inf_{\what BX^{\HH}BX=I_{\what n}}\trace(\what A X^{\HH}AX)
		&=
		\sum_{i=1}^{\inertia_+(\what B)-\properind_+(\what B)} \lambda_i^{+\downarrow}(\what A,\what B)\,\lambda_i^{+\uparrow}(A,B)
		+\sum_{i=1}^{\properind_+(\what B)} \lambda_i^{+\uparrow}(\what A,\what B)\,\lambda_i^{+\downarrow}(A,B)
		\notag\\
		&+\sum_{j=1}^{\properind_-(\what B)} \lambda_j^{-\downarrow}(\what A,\what B)\,\lambda_j^{-\uparrow}(A,B)
		+\sum_{j=1}^{\inertia_-(\what B)-\properind_-(\what B)} \lambda_j^{-\uparrow}(\what A,\what B)\,\lambda_j^{-\downarrow}(A,B)
		.  \label{eq:main-inf}
	\end{align}
	The infimum can be attained, when $(\what A,\what B)$ and $(A,B)$ are congruent-diagonalizable.
	Similarly,  in the second case, the formula for the infimum can be gotten by applying \cref{eq:main-inf} to
	matrix pairs $(-A, -B)$ and $(-\what A, -\what B)$.
\end{theorem}

The three excluded cases in the conditions of the theorem are not interesting:
\begin{enumerate}[(1)]
	\item if $\what A=0$, then
		$\trace(\what A X^{\HH}AX)\equiv 0$ for any $X$;
	\item  if $A=\mu B$ for some $\mu\in\bbR$, then any $X$ such that $\what BX^{\HH}BX=I_{\what n}$
		yields $X^{\HH}BX=\what B^{-1}$, and hence
		\[
			\trace(\what A X^{\HH}AX)=\mu\trace(\what A X^{\HH}BX)\equiv\mu\trace(\what A \what B^{-1});
		\]
	\item  if $\what A=\what \mu\what B$ for some $\what\mu\in \bbR$ when $n=\what n$, then
		any $X$ such that $\what BX^{\HH}BX=I_{\what n}=I_n$ which yields $X\what BX^{\HH}=B^{-1}$ and hence
		\[
			\trace(\what A X^{\HH}AX)=\what\mu\trace(\what B X^{\HH}AX)
			=\what\mu\trace(X\what B X^{\HH}A)\equiv\what\mu\trace(B^{-1}A).
		\]
\end{enumerate}
We will comment on the attainability of the infimum in \cref{eq:main-inf} in a moment.


The proof of this theorem spreads out in the next two sections: first for a special case in \cref{sec:the-simple-case}
and then for the general case of the theorem in \cref{sec:the-general-case}. One common step is to simplify the infimum by
performing congruence transformations to transform $(A,B)$ and $(\what A,\what B)$ into their canonical forms
$(\Lambda, J)$ and $(\what\Lambda, \what J)$, respectively:
\begin{equation}\label{eq:congr-trans}
	B=Y^{\HH}JY,\, A=Y^{\HH}\Lambda Y,
	\quad\text{and}\quad
	\what B=\what Y^{\HH}\what J\what Y,\,\what A=\what Y^{\HH}\what \Lambda \what Y,
\end{equation}
where $Y\in\bbC^{n\times n},\,\what Y\in\bbC^{\what n\times \what n}$ are nonsingular,
and the canonical forms $(\Lambda, J)$ and $(\what\Lambda, \what J)$ will be specified later in the two sections of proofs,
dependent of what assumptions to be made on the pairs.
We have by \cref{eq:congr-trans}
\begin{align*}
	\trace(\what A X^{\HH}AX)
	&=\trace(\what Y^{\HH}\what \Lambda \what Y X^{\HH}Y^{\HH}\Lambda YX)  \\
	&=\trace(\what \Lambda \what Y X^{\HH}Y^{\HH}\Lambda YX\what Y^{\HH})  \\
	&=\trace(\what \Lambda\wtd X^{\HH}\Lambda\wtd X),
\end{align*}
where $\wtd X=YX\what Y^{\HH}$.
Notice also, $\what BX^{\HH}BX=I_{\what n}$ can be turned into
\[
	I_{\what n}=\what BX^{\HH}BX=\what Y^{\HH}\what J\what Y X^{\HH}Y^{\HH}JYX=\what Y^{\HH}\what J\wtd X^{\HH}J\wtd X\what Y^{-\HH}
	\Leftrightarrow  \what J\wtd X^{\HH}J\wtd X=I_{\what n},
\]
with the same  $\wtd X=YX\what Y^{\HH}$ moments ago.
Hence
\begin{equation}\label{eq:simplified}
	\inf_{\what BX^{\HH}BX=I_{\what n}}\trace(\what A X^{\HH}AX)
	=\inf_{\what J\wtd X^{\HH}J\wtd X=I_{\what n}}\trace(\what \Lambda\wtd X^{\HH}\Lambda\wtd X)
	=\inf_{ \what JX^{\HH}J X=I_{\what n}}\trace(\what \Lambda X^{\HH}\Lambda X).
\end{equation}

We now comment on the attainability of the infimum in \cref{eq:main-inf}
when $\Lambda$, $J$, $\wtd\Lambda$, and $\what J$ are diagonal.
Suppose that both $(A, B)$ and $(\what A, \what B)$ are also positive semidefinite pairs,
and $\what B$ is nonsingular (because of $\what BX^{\HH}BX=I_{\what n}$). The other case when both
pairs are nonnegative semidefinite can be handled in the same way.
Since any singularity in $B$ can also be eliminated by congruence transformation (see \Cref{rk:LambdaJ-PD} later),
we may also assume that $B$ is nonsingular as well. So we can write
\begin{subequations}\label{eq:JLambda:the-simple-case}
	\begin{align}
		J=\begin{bmatrix}
			I_{n_+}&\\ & -I_{n_-}
		\end{bmatrix}, \quad
		&\Lambda=\kbordermatrix{ &\sss n_+ &\sss n_- \\
			\sss n_+ &\Lambda_+& \\
		\sss n_- &&-\Lambda_-},  \label{eq:JLambda:the-simple-case:1}\\
		\what J=\begin{bmatrix}
			I_{\what n_+}&\\ & -I_{\what n_-}
		\end{bmatrix}, \quad
		&\what\Lambda=\kbordermatrix{ &\sss \what n_+ &\sss \what n_- \\
			\sss \what n_+ &\what\Lambda_+& \\
		\sss \what n_- &&-\what\Lambda_-}, \label{eq:JLambda:the-simple-case:2}
	\end{align}
\end{subequations}
where $n_++n_-= n$, $\what n_++\what n_-=\what n$, and
$\Lambda,\,\what\Lambda$  are real diagonal matrices.
It can be seen that $\eig(A,B)=\eig(\Lambda,J)=\eig(\Lambda_+)\cup \eig(\Lambda_-)$
and $\eig(\what A,\what B)=\eig(\what\Lambda,\what J)=\eig(\what\Lambda_+)\cup \eig(\what\Lambda_-)$,
where and in the following $\eig(\,\cdot\,)$ and $\eig(\,\cdot\,,\,\cdot\,)$ are the spectrum of a matrix
and that of a matrix pair, respectively.
Each eigenvalue $\lambda^+\in\eig(\Lambda_+)$ is of positive type, i.e., $\bx_+^{\HH}B\bx_+>0$ for its associated eigenvector $\bx_+$,
and each eigenvalue $\lambda^-\in\eig(\Lambda_-)$ is of positive type, i.e., $\bx_-^{\HH}B\bx_-<0$ for its associated eigenvector $\bx_-$. The same can be said about $(\what A,\what B)$.
For $\wtd X=P_{(:,1:\what n)}\what P^{\T}$ where
$P\in\group P_n$ and $P_{(:,1:\what n)}$ stands for the first $\what n$ columns of $P$, and $\what P\in\group P_{\what n}$,
we have
\[
	\trace(\what \Lambda\wtd X^{\HH}\Lambda\wtd X)
	=\trace(\big[\what P^{\T}\what \Lambda \what P\big]\big[P_{(:,1:\what n)}^{\T}\Lambda P_{(:,1:\what n)}\big]).
\]
Hence with $\wtd X=P_{(:,1:\what n)}\what P^{\T}$, $\trace(\what \Lambda\wtd X^{\HH}\Lambda\wtd X)$ is the sum of
products between the diagonal entries of $\what\Lambda$, i.e., the eigenvalues of $(\what A,\what B)$,
and some of the those of $\Lambda$, the eigenvalues of $( A, B)$.
Certainly, there is $\wtd X_{\opt}$ that can be explicitly constructed to
gives the right-hand side of \cref{eq:main-inf}. Observe each product there is for two eigenvalues of the same type:
positive or negative, and hence
$P_{(:,1:\what n)}^{\T}J P_{(:,1:\what n)}=\what P^{\T}\what J \what P$ for that particular $\wtd X_{\opt}$, yielding
\[
	[\what P^{\T}\what J \what P][P_{(:,1:\what n)}^{\T}J P_{(:,1:\what n)}]=I_{\what n}
	\,\Rightarrow\,
	\what J\wtd X_{\opt}^{\HH}J\wtd X_{\opt}=I_{\what n}.
\]
Therefore $\wtd X_{\opt}$ attains the second infimum in \cref{eq:simplified}. Finally,
$X_{\opt}=Y^{-1}\wtd X_{\opt}\what Y^{-\HH}$ attains the infimum in \cref{eq:main-inf}.

\section{The simple case}\label{sec:the-simple-case}
In this section, we prove \Cref{thm:J-unitary:total} for the simple case:
\emph{$n=\what n$, and both pairs $(A,B)$ and $(\what A,\what B)$ are congruent-diagonalizable},
namely, we have \cref{eq:congr-trans} with \cref{eq:JLambda:the-simple-case}
and also $n_++n_-= n$, $\what n_+=n_+$,  $\what n_-=n_-$, where
$\Lambda,\,\what\Lambda$  are real diagonal matrices.
In figuring out \cref{eq:JLambda:the-simple-case}, we
note that necessarily, $\what B$, $B$, and
$X$ are nonsingular because of constraint $\what BX^{\HH}BX=I_{\what n}=I_n$, and that,
by the Sylvester inertia law,
$\inertia_{\pm}(\what B)= \inertia_{\pm}(\what B^{-1})= \inertia_{\pm}(B)$ upon noticing
$X^{\HH}BX=\what B^{-1}$.

We have $J=\what J=J^{-1}$ and hence the last infimum in \cref{eq:simplified} becomes
\begin{equation}\label{eq:simple-simplified}
	\inf_{ X^{\HH}J X=J}\trace(\what \Lambda X^{\HH}\Lambda X).
\end{equation}

When $J=\pm I_n$, both $(A,B)$ and $(\what A,\what B)$ are positive semidefinite pairs because
$A-\lambda_0 B\succ 0$ and $\what A-\lambda_0\what B\succ 0$ for $\lambda_0<0$ with sufficiently large $\abs{\lambda_0}$
if $J=I_n$ or for sufficiently large $\lambda_0\in\bbR$ if $J=-I_n$. Also $(A,B)$ and $(\what A,\what B)$ are negative semidefinite pairs,
too, because
$(-A)-\lambda_0(-B)\succ 0$ and $(-\what A)-\lambda_0(-\what B)\succ 0$ for sufficiently large $\lambda_0\in\bbR$
if $J=I_n$ or for $\lambda_0<0$ with sufficiently large $\abs{\lambda_0}$ if $J=-I_n$.

As for \cref{eq:simple-simplified}, the case when\footnote{When $J=-I_n$, $X^{\HH}JX=J$ becomes $X^{\HH}X=I_n$, the same as for $J=I_n$.}
$J=\pm I_n$ is a known one and has been resolved in the literature, e.g.,~\cite[Theorem~4.3.53]{hojo:2013} as stated in the next lemma.

\begin{lemma}[{\cite[Theorem~4.3.53]{hojo:2013}}]\label{lm:unitary:full}
	Given Hermitian matrices $A_i=U_i\Lambda_iU_i^{\HH}\in\bbC^{n\times n}$
	with $U_i\in \group U_n,\, \Lambda_i\in\group D_n$ for $i=0,1$,
	we have
	\[
		\min_{V\in \group U_n}\trace(A_0VA_1V^{\HH})
		=\min_{V\in \group P_n}\trace(\Lambda _0V\Lambda _1V^{\HH})
		=\sum_{i=1}^n\lambda_i^{\downarrow}(A_0)\lambda_i^{\uparrow}(A_1)
		.
	\]
\end{lemma}

\Cref{lm:unitary:full} can be proved by using an important result on doubly stochastic matrices, namely the
Birkhoff theorem. A matrix $Y\in\bbR^{n\times n}$
is {\em doubly stochastic\/} if entrywise $Y\ge 0$, and $Y\one_n=\one_n$ and $\one_n^{\T}Y=\one_n^{\T}$
where $\one_n\in\bbR^n$ is the vector of all ones. The
Birkhoff theorem says that
the doubly stochastic matrices are convex combinations of permutation
matrices.
Next, we will use this theorem to prove a  result, related to \Cref{lm:unitary:full}, which we will need later.

\begin{lemma}[{\cite{elfr:1995}}]\label{lm:elfr1995}
	Let $X=[x_{ij}]\in\bbC^{n\times n}$ and $Y=[\abs{x_{ij}}^2]\in\bbR^{n\times n}$. Then there exist doubly stochastic matrices $Y_1,\,Y_2\in\bbR^{n\times n}$ such that
	entrywise
	\[
		[\sigma_{\min}(X)]^2\, Y_1 \le Y \le [\sigma_{\max}(X)]^2\, Y_2,
	\]
	where $\sigma_{\min}(X)$ and $\sigma_{\max}(X)$ are the smallest and largest singular values of $X$, respectively.
\end{lemma}

\begin{lemma}\label{lm:gen:full}
	Given positive semidefinite matrices $A_i=U_i\Lambda_iU_i^{\HH}\in\bbC^{n\times n}$
	with $U_i\in \group U_n$, $\Lambda_i\in\group D_n$ for $i=0,1$, we have
	\begin{align*}
		\trace(A_0X^{\HH}A_1X)
		&\le[\sigma_{\max}(X)]^2\max_{V\in \group P_n}\trace(\Lambda _0V\Lambda _1V^{\HH})) \nonumber \\
		&=[\sigma_{\max}(X)]^2\sum_{i=1}^n\lambda_i^{\downarrow}(A_0)\lambda_i^{\downarrow}(A_1),
		\\
		\trace(A_0X^{\HH}A_1X)
		&\ge[\sigma_{\min}(X)]^2\min_{V\in \group P_n}\trace(\Lambda _0V\Lambda _1V^{\HH})) \nonumber \\
		&=[\sigma_{\min}(X)]^2\sum_{i=1}^n\lambda_i^{\downarrow}(A_0)\lambda_i^{\uparrow}(A_1),
	\end{align*}
\end{lemma}

\begin{proof}
	It can be seen that
	\[
		\trace(A_0X^{\HH}A_1X)=\trace(U_0\Lambda_0U_0^{\HH}X^{\HH}U_1\Lambda_1U_1^{\HH}X)
		=\trace(\Lambda_0[U_1^{\HH}XU_0]^{\HH}\Lambda_1[U_1^{\HH}XU_0]).
	\]
	Write $U_1^{\HH}XU_0=[x_{ij}]$ which has the same singular values as $X$ and let $Y=[\abs{x_{ij}}^2]$. We get
	\[
		\trace(A_0X^{\HH}A_1X)=\sum_{i,j=1}^n\lambda_j(A_0)\lambda_i(A_1)\abs{x_{ij}}^2.
	\]
	Noticing that all $\lambda_j(A_0),\,\lambda_i(A_1)\ge 0$.
	Now use \Cref{lm:elfr1995} and the
	Birkhoff theorem to complete the proof, following the standard technique that has been used
	frequently in the matrix eigenvalue perturbation theory \cite{howi:1953,li:1993a}.
\end{proof}

The key tool to analyze the  infimum in \cref{eq:simple-simplified} is the structure of matrix $X\in \bbC^{n\times n}$ satisfying $X^{\HH}JX=J$. Such matrix $X$ is said \emph{$J$-unitary} in literature.

\begin{lemma}[{\cite[Example~6.3]{veselic2011damped}}]\label{lm:polar-decomposition:+}
	Let $J=\diag(I_{n_+},-I_{n_-})$ and $n=n_++n_-$. A matrix $X\in\bbC^{n\times n}$ satisfies $X^{\HH}JX=J$ if and only if it is of the form
	\begin{equation}\label{eq:lm:polar-decomposition:+}
		X=\begin{bmatrix}
			(I_{n_+}+ W  W ^{\HH})^{1/2} &  W \\
			W ^{\HH} & (I_{n_-}+ W ^{\HH} W )^{1/2}\\
			\end{bmatrix}\begin{bmatrix}
			V_+ & \\ & V_-
		\end{bmatrix},
	\end{equation}
	where $V_+\in \group U_{n_+}$, $V_-\in \group U_{n_-}$ , and $W\in \bbC^{n_+\times n_-}$.
\end{lemma}

\Cref{lm:polar-decomposition:+} can be found in \cite{veselic1993jacobi,kove:1995}, where \cref{eq:lm:polar-decomposition:+} is called a (hyperbolic) \emph{polar decomposition} of $X$.
In what follows, we will limit our consideration to the case $n_+\ge n_-\ge 1$,
and the other case $1\le n_+<n_-$ can be handled
in a similar way.

A direct consequence of \Cref{lm:polar-decomposition:+} is \Cref{lm:chsh-decomposition} below, in which \cref{eq:lm:chsh-decomposition} is the so-called \emph{ChSh decomposition} of a $J$-unitary matrix $X$, an analogue of the CS decomposition of a unitary matrix \cite{stsu:1990}.

\begin{lemma}[ChSh Decomposition]\label{lm:chsh-decomposition}
	Let $J=\diag(I_{n_+},-I_{n_-})$ and $n=n_++n_-$, where $n_+\ge n_-$. A matrix $X\in \mathbb C^{n\times n}$ is $J$-unitary
	if and only if it is of the form
	\begin{equation}\label{eq:lm:chsh-decomposition}
		\begin{aligned}[b]
			X&=\begin{bmatrix}
				U_+ & \\ & U_-
				\end{bmatrix}\begin{bmatrix}
				I_{n_+-n_-} &&\\
				&(I_{n_-}+ \Sigma^2)^{1/2} &  \Sigma \\
				&	 \Sigma & (I_{n_-}+ \Sigma^2 )^{1/2}\\
				\end{bmatrix}\begin{bmatrix}
				V_+ & \\ & V_-
			\end{bmatrix}
			\\&=\begin{bmatrix}
				U_+ & \\ & U_-
				\end{bmatrix}\begin{bmatrix}
				(I_{n_+}+ \wtd\Sigma\wtd\Sigma^{\HH} )^{1/2} &  \wtd\Sigma \\
				\wtd\Sigma  & (I_{n_-}+ \wtd\Sigma^{\HH} \wtd\Sigma )^{1/2}\\
				\end{bmatrix}\begin{bmatrix}
				V_+ & \\ & V_-
			\end{bmatrix},
		\end{aligned}
	\end{equation}
	where $U_+,V_+\in \group U_{n_+}$ and $U_-,V_-\in  \group U_{n_-}$,
	$\wtd\Sigma=\begin{bmatrix}
		0\\ \Sigma\\
	\end{bmatrix}\in\bbR^{n_+\times n_-}$ with $\Sigma\in \bbR^{n_-\times n_-}$ being diagonal
	and having nonnegative diagonal entries.
\end{lemma}

\begin{lemma}\label{lm:LambdaJ>=0}
	Let $(\Lambda, J)$ be as in \cref{eq:JLambda:the-simple-case:1} where $n_{\pm}\ge 1$ and
	$\Lambda$ is real diagonal. Then
	\begin{enumerate}[{\rm (i)}]
		\item $(\Lambda,J)\succeq0$ if and only if
			$\lambda_i^+-\lambda_j^-\ge 0$ for any $\lambda_i^+\in \eig(\Lambda_+),\lambda_j^-\in\eig(\Lambda_-)$;
		\item $(\Lambda,J)\preceq0$ if and only if
			$\lambda_i^+-\lambda_j^-\le 0$ for any $\lambda_i^+\in \eig(\Lambda_+),\lambda_j^-\in\eig(\Lambda_-)$.
	\end{enumerate}
\end{lemma}

\begin{proof}
	If $(\Lambda,J)\succeq0$, then there exists $\lambda_0\in\bbR$ such that $\Lambda-\lambda_0J\succeq 0$, i.e.,
	$\lambda_j^-\le\lambda_0\le\lambda_i^+$ for any $\lambda_i^+\in \eig(\Lambda_+),\lambda_j^-\in\eig(\Lambda_-)$
	and thus $\lambda_i^+-\lambda_j^-\ge 0$.
	On the other hand if $\lambda_i^+-\lambda_j^-\ge 0$ for any $\lambda_i^+\in \eig(\Lambda_+),\lambda_j^-\in\eig(\Lambda_-)$,
	then
	\[
		\max\set{\lambda_j^-\,:\,\lambda_j^-\in\eig(\Lambda_-)}
		\le\min\set{\lambda_i^+\,:\,\lambda_i^+\in \eig(\Lambda_+)}
	\]
	and hence any $\lambda_0$ that lies between the maximum and minimum in this inequality makes
	$\Lambda-\lambda_0J\succeq 0$. This proves item (i). For item (ii), by definition,
	$(\Lambda,J)\preceq0$ if and only if $(-\Lambda,-J)\succeq0$, and we then can use item (i).
\end{proof}

With \Cref{lm:chsh-decomposition}, we have
\begin{multline}\label{eq:J-unitary:half:inf:1}
	\inf_{X^{\HH}JX=J
	}\trace(\what \Lambda  X^{\HH}\Lambda X)
	= \inf_{0\preceq\Sigma\in \group D_{n_-}\atop  U_{\pm},V_{\pm}\in \group U_{n_{\pm}}}
	\trace\Bigg(\begin{bsmallmatrix}
			V_+\what \Lambda_+V_+^{\HH} & \\ & -V_-\what \Lambda_-V_-^{\HH}\\
			\end{bsmallmatrix}\begin{bsmallmatrix}
			I&&\\
			&(I + \Sigma^2 )^{1/2} &  \Sigma \\
			&	 \Sigma  & (I + \Sigma^2 )^{1/2}
		\end{bsmallmatrix}\times \\
		\begin{bsmallmatrix}
			U_+^{\HH}\Lambda_+U_+ & \\ & -U_-^{\HH}\Lambda_-U_-\\
			\end{bsmallmatrix}\begin{bsmallmatrix}
			I&&\\
			& (I + \Sigma^2 )^{1/2} &  \Sigma \\
			& \Sigma  & (I + \Sigma^2 )^{1/2}\\
	\end{bsmallmatrix}  \Bigg)
\end{multline}

By \Cref{lm:LambdaJ>=0}, if
$(\Lambda,J)$ and $(\what\Lambda,J)$ are not both positive semidefinite pairs, or not both negative definite pairs,
then\footnote {Besides the condition just mentioned, this claim also requires the condition given in the theorem:
	$A\ne \mu B$ for any $\mu\in \bbR$ and $\what A\ne\what \mu\what B$ for any $\what\mu\in \bbR$, 
	which is equivalent to $\Lambda\ne \mu J$ for any $\mu\in \bbR$ and $\what \Lambda\ne\what \mu J$ for any $\what\mu\in \bbR$ because of \cref{eq:congr-trans}. Otherwise if $\Lambda=\lambda_0 J$ for some $\lambda_0\in\bbR$, then
	$(\Lambda,J)\succeq 0$ and $\lambda_i^+-\lambda_j^-=0$ for any $\lambda_i^+\in \eig(\Lambda_+)$ and $\lambda_j^-\in\eig(\Lambda_-)$ and hence
	$\big(\what\lambda_{\what i}^+-\what\lambda_{\what j}^-\big)(\lambda_i^+-\lambda_j^-)=0$,
	regardless whether $(\what\Lambda,J)\succeq 0$ or not.
}
there exist $\what \lambda_{\what i}^+\in \eig(\what \Lambda_+)$, $\what \lambda_{\what j}^-\in\eig(\what \Lambda_-)$,
$\lambda_i^+\in \eig(\Lambda_+)$, and $\lambda_j^-\in\eig(\Lambda_-)$
with
\[
	\big(\what\lambda_{\what i}^+-\what\lambda_{\what j}^-\big)(\lambda_i^+-\lambda_j^-)<0.
\]
We now restrict $\Sigma$, $U_{\pm}$, and $V_{\pm}$ in \cref{eq:J-unitary:half:inf:1} to special ones and doing so will
increase the infimum there. Specifically, we let
$\Sigma=\sigma \be_1\be_1^{\T}$
where $\sigma$ is free to vary and $\be_1$ is the first column of $I$ of apt size, and
$V_{\pm}$ and $U_{\pm}$ as products of permutation matrices
\[
	V_+=P_{2+}\what P_{1+}^{\HH}, \quad
	V_-=P_{2-}\what P_{1-}^{\HH}, \quad
	U_+=P_{1+}P_{2+}^{\HH}, \quad
	U_+=P_{1-}P_{2-}^{\HH}
\]
such that
\begin{gather*}
	\big[\what P_{1+}^{\HH}\what\Lambda_+\what P_{1+}\big]_{(1,1)}=\what\lambda_{\what i}^+, \quad
	\big[P_{1+}^{\HH}\Lambda_+P_{1+}\big]_{(1,1)}=\lambda_{i}^+, \\
	\big[\what P_{1-}^{\HH}\what\Lambda_-\what P_{1-}\big]_{(1,1)}=\what\lambda_{\what j}^-, \quad
	\big[P_{1-}^{\HH}\Lambda_-P_{1-}\big]_{(1,1)}=\lambda_{j}^-, \\
	\big[P_{2+}^{\HH}\diag\big(I,(I + \Sigma^2 )^{1/2}\big)P_{2+}\big]_{(1,1)}=(1+\sigma^2)^{1/2}, \\
	\big[P_{2-}^{\HH}(I + \Sigma^2 )^{1/2}P_{2-}\big]_{(1,1)}=(1+\sigma^2)^{1/2},
\end{gather*}
where $[\cdots]_{(1,1)}$ is the $(1,1)$th entry of a matrix.
%
%
We get from \cref{eq:J-unitary:half:inf:1}
\begin{align}
	\MoveEqLeft[0]\inf_{X^{\HH}JX=J
	}\trace(\what \Lambda  X^{\HH}\Lambda X)
	\notag\\
	&\le \inf_{\sigma>0}\trace\left(\begin{bsmallmatrix}
			\what \lambda^+_{\what i} & \\ & *\\ & & -\what \lambda^-_{\what j} \\ &&&*\\
			\end{bsmallmatrix}\begin{bsmallmatrix}
			(1 + \sigma^2 )^{1/2} & &  \sigma & \\
			& I & & 0\\
			\sigma &  & (1 + \sigma^2 )^{1/2} &\\
			& 0 && I\\
			\end{bsmallmatrix}\begin{bsmallmatrix}
			\lambda_{i}^+ & \\ & *\\ & & -\lambda_{j}^- \\ &&&*\\
			\end{bsmallmatrix}\begin{bsmallmatrix}
			(1 + \sigma^2 )^{1/2} & &  \sigma & \\
			& I & & 0\\
			\sigma &  & (1 + \sigma^2 )^{1/2} &\\
			& 0 && I\\
	\end{bsmallmatrix}\right)
	\notag\\
	&=\inf_{\sigma>0}\trace\left(\begin{bsmallmatrix}
			\what \lambda_{\what i}^+ & \\ & -\what \lambda_{\what j}^- \\
			\end{bsmallmatrix}\begin{bsmallmatrix}
			(1 + \sigma^2 )^{1/2} & \sigma  \\
			\sigma &  (1 + \sigma^2 )^{1/2} \\
			\end{bsmallmatrix}\begin{bsmallmatrix}
			\lambda_{i}^+ & \\ & -\lambda_{j}^- \\
			\end{bsmallmatrix}\begin{bsmallmatrix}
			(1 + \sigma^2 )^{1/2} & \sigma  \\
			\sigma &  (1 + \sigma^2 )^{1/2} \\
	\end{bsmallmatrix}\right)+\text{(constant)}
	\notag\\
	&=\inf_{\sigma>0}\,\big(\what\lambda_{\what i}^+-\what\lambda_{\what j}^-\big)(\lambda_i^+-\lambda_j^-)\,\sigma^2
	+\text{(constant)}    \notag\\
	&= -\infty		.  \label{eq:J-unitary:half:inf}
\end{align}

Suppose now that  $(\Lambda,J)$ and $(\what \Lambda,J)$ are both positive semidefinite or both negative semidefinite.
Since we can switch to considering
$(-\Lambda,-J)$ and $(-\what \Lambda,-J)$ instead if both $(\Lambda,J)$ and $(\what \Lambda,J)$ are  negative semidefinite, it suffices to consider both $(\Lambda,J)$ and $(\what \Lambda,J)$ are positive semidefinite only, which we now assume.
Then there exist $\lambda_0,\what \lambda_0\in\bbR$ such that
$\Lambda-\lambda_0 J\succeq 0$ and $\what\Lambda-\what\lambda_0J\succeq 0$.
With $X^{\HH}JX=J$, we have
\begin{align}
	\trace(\what \Lambda X^{\HH}\Lambda X)
	&= \trace(\what \Lambda X^{\HH}[\Lambda-\lambda_0 J] X)+\lambda_0\trace(\what \Lambda J
	)
	\nonumber\\
	&= \trace([\what \Lambda-\what \lambda_0J
	] X^{\HH}[\Lambda-\lambda_0 J] X)+\what\lambda_0\trace(J
	X^{\HH}[\Lambda-\lambda_0 J] X)+\lambda_0\trace(\what \Lambda J
	)
	\nonumber\\
	&= \trace([\what \Lambda-\what \lambda_0J
	] X^{\HH}[\Lambda-\lambda_0 J] X)+\what\lambda_0\trace((JX)^{-1}[\Lambda-\lambda_0 J] X)+\lambda_0\trace(\what \Lambda J
	)
	\nonumber\\
	&= \trace([\what \Lambda-\what \lambda_0J
	] X^{\HH}[\Lambda-\lambda_0 J] X)+\what\lambda_0\trace(J\Lambda )-\what\lambda_0\lambda_0\trace(I)+\lambda_0\trace(\what \Lambda J
	)
	, \label{eq:inf(shiftLambda)}
\end{align}
where only the first term varies with $X$. Since $\Lambda-\lambda_0 J\succeq 0$ and $\what\Lambda-\what\lambda_0J\succeq 0$,
it effectively transforms the problem into the case when both $\what \Lambda$ and $\Lambda$ are positive semidefinite, i.e.,
$\Lambda_+,\,\what\Lambda_+\succeq 0$ and $-\Lambda_-,\,-\what\Lambda_-\succeq 0$,
which we will assume for the moment.
Then,
\begin{align*}
	\MoveEqLeft[0]\inf_{X^{\HH}JX=J
	}\trace(\what \Lambda  X^{\HH}\Lambda X)
	\\
	&= \inf_{0\preceq\Sigma\in \group D_{n_-}\atop  U_{\pm},V_{\pm}\in \group U_{n_{\pm}}}
	\trace\Bigg(\begin{bsmallmatrix}
			V_+\what \Lambda_+V_+^{\HH} & \\ & -V_-\what \Lambda_-V_-^{\HH}\\
			\end{bsmallmatrix}\begin{bsmallmatrix}
			(I + \wtd\Sigma\wtd\Sigma^{\HH} )^{1/2} &  \wtd\Sigma \\
			\wtd\Sigma  & (I + \wtd\Sigma^{\HH}\wtd\Sigma )^{1/2}\\
		\end{bsmallmatrix} \times \\
		&\hspace{3cm}\begin{bsmallmatrix}
			U_+^{\HH}\Lambda_+U_+ & \\ & -U_-^{\HH}\Lambda_-U_-\\
			\end{bsmallmatrix}\begin{bsmallmatrix}
			(I + \wtd\Sigma\wtd\Sigma^{\HH} )^{1/2} &  \wtd\Sigma \\
			\wtd\Sigma  & (I + \wtd\Sigma^{\HH}\wtd\Sigma )^{1/2}\\
	\end{bsmallmatrix}  \Bigg)
	\\
	&= \inf_{0\preceq\Sigma\in \group D_{n_-}\atop  U_{\pm},V_{\pm}\in \group U_{n_{\pm}}}
	\Bigg[
		\underbrace{\trace\Big( V_+\what \Lambda_+V_+^{\HH}(I + \wtd\Sigma\wtd\Sigma^{\HH} )^{1/2}U_+^{\HH}\Lambda_+U_+(I + \wtd\Sigma\wtd\Sigma^{\HH} )^{1/2}\Big)}_{=:\tau_1} \\
		&\hspace{3cm}+\underbrace{\trace\Big( V_+\what \Lambda_+V_+^{\HH} \wtd\Sigma^{\HH} U_-^{\HH}[-\Lambda_-]U_-\wtd\Sigma\Big)}_{=:\tau_2}
		+\underbrace{\trace\Big( V_-[-\what \Lambda_-]V_-^{\HH}\wtd\Sigma U_+^{\HH}\Lambda_+U_+\wtd\Sigma^{\HH}\Big)}_{=:\tau_3} \\
		&\hspace{3cm}+\underbrace{\trace\Big(V_-[-\what \Lambda_-]V_-^{\HH}(I + \wtd\Sigma^{\HH}\wtd\Sigma )^{1/2}U_-^{\HH}[-\Lambda_-]U_-(I + \wtd\Sigma^{\HH}\wtd\Sigma )^{1/2}\Big)}_{=:\tau_4}
	\Bigg]
	\\&\ge
	\inf_{0\preceq\Sigma\in \group D_{n_-}\atop  U_+,V_+\in \group U_{n_+}}\tau_1
	+\inf_{0\preceq\Sigma\in \group D_{n_-}\atop  U_-\in \group U_{n_-},V_+\in \group U_{n_+}}\tau_2
	+\inf_{0\preceq\Sigma\in \group D_{n_-}\atop  U_+\in \group U_{n_+},V_-\in \group U_{n_-}}\tau_3
	+\inf_{0\preceq\Sigma\in \group D_{n_-}\atop  U_-,V_-\in \group U_{n_-}}\tau_4
	.
\end{align*}
Next we bound these four infima from below. By \Cref{lm:gen:full} (with $X=(I + \wtd\Sigma\wtd\Sigma^{\HH} )^{1/2}$),
we have
\begin{align*}
	\inf_{0\preceq\Sigma\in \group D_{n_-}\atop  U_+,V_+\in \group U_{n_+}}	\tau_1	
	&\ge \inf_{ U_{\pm},V_{\pm}\in \group P_n}\trace( V_+\what \Lambda_+V_+^{\HH}U_+^{\HH}\Lambda_+U_+)
	\\&= \sum_{i=1}^{n_+} \lambda_i^{\downarrow}(\what \Lambda_+)\lambda_i^{\uparrow}(\Lambda_+)
	,
\end{align*}
and, again by  \Cref{lm:gen:full} (with $X=\begin{bsmallmatrix}
	0 & \\ & \Sigma	
\end{bsmallmatrix}$),
\begin{align*}
	\inf_{0\preceq\Sigma\in \group D_{n_-}\atop  U_-\in \group U_{n_-},V_+\in \group U_{n_+}}\tau_2
	&\ge
	\inf_{0\preceq\Sigma\in \group D_{n_-}\atop  \wtd U_-,V_+\in \group U_{n_+}}\trace( V_+\what \Lambda_+V_+^{\HH} \begin{bsmallmatrix}
		0 & \\ & \Sigma	
		\end{bsmallmatrix} \wtd U_-^{\HH}\begin{bsmallmatrix}
		0 & \\ &-\Lambda_-
		\end{bsmallmatrix}\wtd U_-\begin{bsmallmatrix}
		0 & \\ & \Sigma	
	\end{bsmallmatrix})
	\\
	&\ge 0.
\end{align*}
Similarly, we can bound $\tau_3$ and $\tau_4$ from below.
Put all together to get
\begin{align*}
	\inf_{X^{\HH}JX=J
	}\trace(\what \Lambda  X^{\HH}\Lambda X)
	\ge \sum_{i=1}^{n_+} \lambda_i^{\downarrow}(\what \Lambda_+)\lambda_i^{\uparrow}(\Lambda_+)
	+\sum_{j=1}^{n_-} \lambda_j^{\downarrow}(\what \Lambda_-)\lambda_j^{\uparrow}(\Lambda_-)	.
\end{align*}
Since the right-hand side is achieved by $\trace(\what \Lambda  X^{\HH}\Lambda X)$ at $\Sigma=0$, and $U_{\pm},V_{\pm}\in \group P_{n_{\pm}}$ such that the diagonal values of $\Lambda_{\pm}$ and $\what \Lambda_{\pm}$ are in the increasing and decreasing order respectively, we conclude that
\[
	\min_{X^{\HH}JX=J
	}\trace(\what \Lambda  X^{\HH}\Lambda X)
	=
	\sum_{i=1}^{n_+} \lambda_i^{\downarrow}(\what \Lambda_+)\lambda_i^{\uparrow}(\Lambda_+)+\sum_{j=1}^{n_-} \lambda_j^{\downarrow}(\what \Lambda_-)\lambda_j^{\uparrow}(\Lambda_-)
	.
\]
For general positive semidefinite pairs $(\Lambda,J),(\what \Lambda,J)$, we can apply what we just proved to the first term
in \cref{eq:inf(shiftLambda)} and then simplify.

We summarize what we just proved into \Cref{lm:J-unitary:equal}.

\begin{lemma}\label{lm:J-unitary:equal}
	Given Hermitian matrix pairs $(A,B)$ and $(\what A,\what B)$ with nonsingular $B,\what B\in \bbC^{n\times n}$,
	suppose that both pairs are congruent-diagonalizable and that
	$A\ne \mu B$ for any $\mu\in \bbR$, and $\what A\ne\what \mu\what B$ for any $\what\mu\in \bbR$. 
	Then
	\[
		\inf\limits_{\what BX^{\HH}BX=I_{n}}\trace(\what \Lambda  X^{\HH}\Lambda X)>-\infty
	\]
	if and only if either both $(A,B)$ and $(\what A,\what B)$ are positive semidefinite pairs
	or both are negative semidefinite pairs.
	Moreover, in the first case, i.e., when both $(A,B)$ and $(\what A,\what B)$ are positive semidefinite pairs,
	\begin{equation}\label{eq:J-unitary:equal}
		\min_{\what BX^{\HH}BX=I_{\what n}}\trace(\what A X^{\HH}AX)
		=
		\sum_{i=1}^{\inertia_+(B)} \lambda_i^{+\downarrow}(\what A,\what B)\lambda_i^{+\uparrow}(A,B)+\sum_{j=1}^{\inertia_-(B)} \lambda_j^{-\downarrow}(\what A,\what B)\lambda_j^{-\uparrow}(A,B)
		.
	\end{equation}
	A similar expression for the infimum can be gotten for the case when
	both $(A,B)$ and $(\what A,\what B)$ are negative semidefinite pairs by applying \cref{eq:J-unitary:equal}
	to $(-A,-B)$ and $(-\what A,-\what B)$.
\end{lemma}

\Cref{lm:J-unitary:equal} is a special case of \Cref{thm:J-unitary:total}, and it with
$B=\what B=I_n$ yields \Cref{lm:unitary:full}.

\section{The general case}\label{sec:the-general-case}
In this section we prove \Cref{thm:J-unitary:total} in its generality.
We will assume that $B$ is genuinely indefinite, except in \Cref{rk:definite} later where we will comment on
how the case when $B$ is positive or negative semidefinite can be handled in a simpler way.

We still have the decompositions in \cref{eq:congr-trans} and  simplification in \cref{eq:simplified}, with $(\Lambda,J)$ to be specified
as in \Cref{lm:HCF} and similarly for $(\what\Lambda,\what J)$.
%

\begin{lemma}[{\cite[Theorem 5.1]{laro:2005}}]\label{lm:HCF}
	Let $p$ be a positive integer and
	\begin{equation*}\label{eq:HCF:block}
		K_p(\tau)=
		\begin{bmatrix}
			&   &         &   & & \tau \\
			&   &         &   & \tau & 1 \\
			&   &         & \iddots  & 1 &  \\
			&   & \iddots & \iddots  &  \\
			& \tau &  1       &   &  \\
			\tau & 1  &         &   &
		\end{bmatrix}_{p\times p},
		\qquad
		F_p=
		\begin{bmatrix}
			&   &   &      &   & 1 \\
			&   &  &       & 1 &  \\
			&   & &\iddots &   &  \\
			&   & \iddots &   &  \\
			& 1 &         &&   &  \\
			1 &   &       &  &   &
		\end{bmatrix}_{p\times p}.
	\end{equation*}
	Any Hermitian matrix pair $(A,B)$ is congruent to $(\Lambda,J)$ as a direct sum of possible block pairs of types:
	\begin{description}
		\item[T-o:] $(0,0)$,
		\item[T-s($2p+1$):] $\left(K_{2p+1}(0),\begin{bmatrix}
					&& F_{p}\\
					&0 &\\
					F_p &&\\
			\end{bmatrix}\right)$,
		\item[T-$\infty$($p$):] $(\eta F_p, \eta K_p(0))$ with $\eta\in\set{\pm1}$, associated with an infinite eigenvalue,
		\item[T-c($p$):] $(\begin{bmatrix}
				0 & K_p(\alpha+\ii\beta)\\ K_p(\alpha-\ii\beta) & \\
			\end{bmatrix},F_{2p})$, associated with a pair of genuine conjugate complex eigenvalues
			$\alpha\pm\ii\beta$ with $\alpha\in \mathbb{R}, \beta>0$,
		\item[T-r($p$):] $(\eta K_p(\alpha), \eta F_p)$  with $\eta\in\set{\pm1}$, 
			associated with a finite real eigenvalue $\alpha$.
	\end{description}
	Moreover, $(\Lambda,J)$ is unique
	up to a simultaneous permutation of the corresponding diagonal block pairs.
\end{lemma}

Although \Cref{lm:HCF} lists five possible types of block pairs that each of $(\Lambda, J)$ and $(\what\Lambda,\what J)$ may contain, we can quickly exclude
some types of block pairs from $(\Lambda, J)$ and $(\what\Lambda,\what J)$ based on the nature of our problem.
\begin{itemize}
	\item $(\what \Lambda,\what J)$ possibly contains
		block pairs of types \block{c}{$p$} and \block{r}{$p$} only.
		This is because $\what B$ is nonsingular and so is $\what J$, by \Cref{lm:HCF}, and hence
		block pairs of type \blocko{}, \block{s}{$2p+1$}, or \block{$\infty$}{$p$} do not
		show up in pair $(\what \Lambda,\what J)$.  For that reason,
		we will have $\what J^{-1}=\what J$ and hence constraint $\what JX^{\HH}JX=I_{\what n}$ is equivalent to
		$X^{\HH}JX=\what J$. It follows from \cref{eq:simplified} that
		\begin{equation}\label{eq:simplified'}
			\inf_{\what BX^{\HH}BX=I_{\what n}}\trace(\what A X^{\HH}AX)
			=\inf_{ \what JX^{\HH}J X=I_{\what n}}\trace(\what \Lambda X^{\HH}\Lambda X)
			=\inf_{ X^{\HH}J X=\what J}\trace(\what \Lambda X^{\HH}\Lambda X).
		\end{equation}
		In the rest of this section, we will investigate the last infimum in \cref{eq:simplified'}.
	\item We can also exclude block pairs of type \blocko{} from $(\Lambda, J)$. In fact, if  $(\Lambda, J)$
		contains block pairs of type \blocko{}, then we can write
		$\Lambda=\begin{bmatrix}
			\Lambda_r & \\ & 0
			\end{bmatrix},J=\begin{bmatrix}
			J_r & \\ & 0
		\end{bmatrix}$,
		and partition $X=\begin{bmatrix}
			X_r \\ X_s
		\end{bmatrix}$ accordingly
		to get
		\[
			\inf_{X^{\HH}JX=\what J}\trace(\what \Lambda  X^{\HH}\Lambda X)
			=\inf_{X_r^{\HH}J_rX_r=\what J}\trace(\what \Lambda  X_r^{\HH}\Lambda_r X_r),
		\]
		which falls into the case that $(\Lambda,J)$ contains no  block pair of type \blocko{}.
\end{itemize}
In summary, possible types of  block pairs to consider forward in $(\Lambda, J)$ and $(\what \Lambda,\what J)$
are
\begin{subequations}\label{eq:BP-possible}
	\begin{align}
		(\Lambda, J):&\quad\text{\block{s}{$2p+1$}, \block{$\infty$}{$p$}, \block{c}{$p$}, \block{r}{$p$}};
		\label{eq:BP-possible-1}\\
		(\what \Lambda,\what J): &\quad\text{\block{c}{$p$}, \block{r}{$p$}}. \label{eq:BP-possible-2}
	\end{align}
\end{subequations}
In our later analysis, we will also replace any  block pair of type \block{c}{$p$} with
\begin{equation}\label{eq:T-c(p)->diagB}
	\left(\begin{bmatrix}
			K_p(\alpha) & -\ii\beta F_p\\ \ii\beta F_p & -K_p(\alpha)
			\end{bmatrix},\begin{bmatrix}
			F_p & \\ & -F_p
	\end{bmatrix}\right).
\end{equation}
This is because they are congruent:
\begin{align*}
	\frac{1}{\sqrt 2}\begin{bmatrix}
		I & I \\ I & -I
		\end{bmatrix}\cdot\begin{bmatrix}
		0 & K_p(\alpha+\ii\beta)\\ K_p(\alpha-\ii\beta) & \\
		\end{bmatrix}\cdot\frac{1}{\sqrt 2}\begin{bmatrix}
		I & I \\ I & -I
	\end{bmatrix}
	&
	=\begin{bmatrix}
		K_p(\alpha) & -\ii\beta F_p\\ \ii\beta F_p & -K_p(\alpha)
	\end{bmatrix}
	,\\
	\frac{1}{\sqrt 2}\begin{bmatrix}
		I & I \\ I & -I
		\end{bmatrix}\cdot\begin{bmatrix}
		& F_p\\ F_p&
		\end{bmatrix}\cdot\frac{1}{\sqrt 2}\begin{bmatrix}
		I & I \\ I & -I
	\end{bmatrix}
	&
	=\begin{bmatrix}
		F_p & \\ & -F_p
	\end{bmatrix}
	.
\end{align*}

%

\begin{remark}\label{rk:LambdaJ-PD}
	When $(A,B)$ is positive semidefinite, possible block pairs in its canonical form are considerably limited
	\cite[Lemma~3.8]{lilb:2013}.
	In fact, if $A-\lambda_0 B\succeq 0$ for some $\lambda_0\in\bbR$, then its canonical form possibly contains
	(0,0) of type \blocko{}, $(\eta K_1(\alpha),\eta F_1)$ of type \block{r}{$1$} such that
	$\eta(\alpha-\lambda_0)\ge 0$,
	$(K_2(\lambda_0), F_2)$ of type \block{r}{$2$},
	$(1,0)$ of type \block{$\infty$}{$1$}.
\end{remark}

The next lemma will be used in \cref{ssec:BP:case1} to reduce the case $n>\what n$ 
to the case  $n=\what n$. It may be of interest in its own and it also sheds light on why 
\Cref{dfn:proper} reads the way it is.

\begin{lemma}\label{lm:NonSq->Sq}
	Let $B\in\bbC^{n\times n},\,\what A,\,\what B\in\bbC^{\what n\times\what n}$ be Hermitian matrices.
	Suppose that $B$ and $\what B$ are nonsingular, and
	$\what n_{\pm}:=\inertia_{\pm}(\what B)\le n_{\pm}:=\inertia_{\pm}(B)$, and let
	\[
		\wtd A=\begin{bmatrix}
			\what A & \\ & 0
		\end{bmatrix}\in\bbC^{n\times n}, \quad
		J_c=\begin{bmatrix}
			I_{n_+-\what n_+} & \\ & -I_{n_--\what n_-}
		\end{bmatrix}, \quad
		\wtd B=\begin{bmatrix}
			\what B & \\ & J_c
		\end{bmatrix}.
	\]
	\begin{enumerate}[{\rm (i)}]
		\item If $(\wtd A,\wtd B)\succeq 0$, then $(\what A,\what B)\succeq 0$ and
			$(B,\what A,\what B)$ is proper; conversely, if $(\what A,\what B)\succeq 0$ and
			$(B,\what A,\what B)$ is proper, then $(\wtd A,\wtd B)\succeq 0$.
		\item If $(\wtd A,\wtd B)\preceq 0$, then $(\what A,\what B)\preceq 0$ and
			$(-B,-\what A,-\what B)$ is proper; conversely, if $(\what A,\what B)\preceq 0$ and
			$(-B,-\what A,-\what B)$ is proper, then $(\wtd A,\wtd B)\preceq 0$.
	\end{enumerate}
\end{lemma}

\begin{proof}
	We will prove item (i) only. Item (ii) becomes item (i) by simply considering $(-\what A,-\what B)$ instead.
	No proof is necessary if $\what n=n$. Suppose that $\what n<n$. There are three subcases to consider:
	(1) both $\what n_{\pm}<n_{\pm}$, (2) $\what n_+<n_+$ and $\what n_-=n_-$, and (3) $\what n_+=n_+$ and $\what n_-<n_-$.

	Consider subcase (1). Suppose that $(\wtd A,\wtd B)\succeq 0$, i.e., $\wtd A-\lambda_0\wtd B\succeq 0$.
	Then $\what A-\lambda_0\what B\succeq 0$ and $-\lambda_0J_c\succeq 0$, implying
	$(\what A,\what B)\succeq 0$ and $\lambda_0=0$. That $\lambda_0=0$ implies that 
	$\lambda_1^{-\downarrow}(\what A,\what B)\le 0\le \lambda_1^{+\uparrow}(\what A,\what B)$ and thus $(B,\what A,\what B)$ is proper.
	Conversely, if $(\what A,\what B)\succeq 0$ and $(B,\what A,\what B)$ is proper, then by
	\Cref{dfn:proper}, we find that 
	$\lambda_1^{-\downarrow}(\what A,\what B)\le 0\le \lambda_1^{+\uparrow}(\what A,\what B)$ and hence
	$\what A-\lambda_0\what B\succeq 0$ for $\lambda_0=0$, i.e., $\what A\succeq 0$, and hence
	$\wtd A-\lambda_0\wtd B\succeq 0$, i.e., $(\wtd A,\wtd B)\succeq 0$.

	Consider subcase (2). Suppose that $(\wtd A,\wtd B)\succeq 0$, i.e., $\wtd A-\lambda_0\wtd B\succeq 0$.
	Then $\what A-\lambda_0\what B\succeq 0$ and $-\lambda_0J_c\succeq 0$, implying
	$(\what A,\what B)\succeq 0$ and $\lambda_0\le 0$.
	Hence $\lambda_1^{-\downarrow}(\what A,\what B)\le\lambda_0\le 0$ and
	$(B,\what A,\what B)$ is proper.
	Conversely, if $(\what A,\what B)\succeq 0$ and $(B,\what A,\what B)$ is proper, then by
	\Cref{dfn:proper}, we find that $\lambda_1^{-\downarrow}(\what A,\what B)\le 0$. By \cref{eq:PSD-intvl},
	$\what A-\lambda_0\what B\succeq 0$ for some $\lambda_0\le 0$ and hence $-\lambda_0J_c\succeq 0$ and
	$\wtd A-\lambda_0\wtd B\succeq 0$, i.e., $(\wtd A,\wtd B)\succeq 0$.

	Subcase (3) can be handled in the same way as handling subcase (2).
\end{proof}

We now prove \Cref{thm:J-unitary:total} in an order of increasing complexity
of $(\Lambda, J)$ and $(\what\Lambda,\what J)$ in terms of possible combinations of block pairs of types
listed in \cref{eq:BP-possible}, and
hence conclude the proof at the end.

\subsection{Involving block pairs of type \block{r}{$1$},\block{c}{$1$} only}
\label{ssec:BP:case1}
In this  case, we have
\begin{equation*}
	J=\begin{bmatrix}
		I_{n_+} &  \\
		& -I_{n_-}
	\end{bmatrix}, \quad
	\what J=\begin{bmatrix}
		I_{\what n_+} &  \\
		& -I_{\what n_-}
	\end{bmatrix}.
\end{equation*}
Recall $\what n_+\le n_+$ and $\what n_-\le n_-$ by \cref{eq:inertia(BtB)}.
Let $J_c=\begin{bmatrix}
	I_{n_+-\what n_+} & \\ & -I_{n_--\what n_-}
\end{bmatrix}$.
For any $X$ such that $X^{\HH}JX=\what J$, we can
complement $X$ to a square matrix $\wtd X=\begin{bmatrix}
	X & X_c
\end{bmatrix}$ such that $\wtd X^{\HH}J\wtd X=\diag(\what J,J_c)$ and then $(\wtd XP)^{\HH}J(\wtd XP)=J$ upon
permuting the columns of $\wtd X$ by some permutation matrix $P$. This is guaranteed by \Cref{lm:complement-basis:+} below
that can be found in many classical monographs, e.g., \cite{malcev1963foundation,gohbergLR2005indefinite}.

\begin{lemma}[{\cite[Corollary~5.12]{veselic2011damped}}]\label{lm:complement-basis:+}
	Let $J=\diag(I_{n_+},-I_{n_-})$ and $n=n_++n_-$. Any set vectors $\bu_1,\dots,\bu_k$ satisfying
	$\bu_i^{\HH}J\bu_j=\pm \delta_{ij}$ for $i,j=1,\dots,k$ can be complemented to a basis $\set{\bu_1,\dots,\bu_n}$ of $\bbC^n$ satisfying $\bu_i^{\HH}J\bu_j=\pm\delta_{ij}$ for $i,j=1,\dots,n$, where $\delta_{ij}$ is the Kronecker delta which is $1$ for $i=j$ and $0$ otherwise, and the numbers of $1$ and $-1$ among $\bu_i^{\HH}J\bu_i$ for $1\le i\le n$ are $n_+$ and $n_-$, respectively.
\end{lemma}

Set
\begin{equation}\label{eq:wtdLambdaJ}
	\wtd\Lambda=\begin{bmatrix}
		\what \Lambda & \\ & 0
	\end{bmatrix}\in\bbC^{n\times n}, \quad
	\wtd J=\begin{bmatrix}
		\what J & \\ & J_c
	\end{bmatrix}.
\end{equation}
It can be seen that
\begin{align*}
	\inf_{\wtd X^{\HH}J\wtd X=\wtd J}\trace(\wtd\Lambda\wtd X^{\HH}\Lambda \wtd X)
	&=\inf_{ X^{\HH}J X=\what J\atop X_c^{\HH}JX=0, X_c^{\HH}JX_c=J_c}\trace(\what \Lambda X^{\HH}\Lambda X)
	=\inf_{ X^{\HH}J X=\what J}\trace(\what \Lambda X^{\HH}\Lambda X).
\end{align*}
This and \Cref{lm:NonSq->Sq} show that we can consider $(\Lambda,J)$ and $(\wtd\Lambda,\wtd J)$ instead, 
for which $n=\what n$.

In the rest of this subsection, we will assume $n=\what n$. We consider three subcases:
\begin{enumerate}[(1)]
	\item Only  block pairs of type \block{r}{$1$} are involved; 
	\item $\what \Lambda= \what\mu\what J$ for some $\what \mu\in \bbR$;
	\item Besides possibly  block pairs of type \block{r}{$1$}, at least one  block pair of type \block{c}{$1$}
		is also involved and  $\what \Lambda\ne \what\mu\what J$ for any $\what \mu\in \bbR$.
\end{enumerate}
Subcase (1) has already been taken care of in \cref{sec:the-simple-case}.
Subcase (2) falls into the excluded cases of the theorem: $\what A\ne \what\mu\what B$ for any $\what \mu\in \bbR$ 
if $n=\what n$ to begin with, i.e., without
the expansions in \cref{eq:wtdLambdaJ}, or if with the expansions then 
$0=\what\mu J_c\,\Rightarrow\,\what\mu=0$, yielding $\what A=0$.

We now turn our attention
to subcase (3).
Now $J,\,\what J\in\bbC^{n\times n}$ are nonsingular, and $\inertia_{\pm}(J)=\inertia_{\pm}(\what J)$.
Notice that the direct sum of pairs of type \block{r}{$1$} is a diagonal pair, and
each block pair of type \block{c}{$1$} can be turned into \cref{eq:T-c(p)->diagB} for $p=1$ by a congruent transformation.
Thus we can assume
\begin{equation*}\label{eq:new-form}
	\Lambda=\begin{bmatrix}\Lambda_+ ^c  & -\ii  \Omega^c  \\ \ii  \Omega^c & -\Lambda_- ^c \end{bmatrix}, \,\,
	J=\begin{bmatrix}I_{n_+ } & \\ & -I_{n_- }\end{bmatrix}, \,\,
	\what \Lambda =\begin{bmatrix}\what \Lambda _+ ^c  & -\ii  \what{\Omega}^c  \\ \ii  \what{\Omega}^c  & -\what \Lambda _- ^c \end{bmatrix}, \,\,
	\what J=J, 
\end{equation*}
where  $\Lambda_{\pm}^c,\what \Lambda_{\pm}^c\in\group D$, and $\Omega^c,,\what \Omega^c\in\bbR^{n_+\times n_-}$ 
are leading diagonal matrices with nonnegative diagonal entries.
As a result,
\[
	\inf_{X^{\HH}JX=\what J}\trace(\what \Lambda  X^{\HH}\Lambda X)
	=\inf_{\Sigma\in\group D^+\atop U_{\pm},V_{\pm}\in \group U}\trace(\what \Lambda  X^{\HH}\Lambda X)
	\le \inf_{\Sigma\in\group D^+\atop U_{\pm},V_{\pm}\in \group P^u}\trace(\what \Lambda  X^{\HH}\Lambda X)
	.
\]
In a way similar to that in \cref{eq:J-unitary:half:inf}, we will
select concrete $U_{\pm}, V_{\pm} \in \group P^u$ to establish a necessary condition such
that the infimum is not $-\infty$.

First we consider the case $n=\what n=2$.
Note that
\[
	\begin{bmatrix}
		\sqrt{1+\sigma^2}&\sigma\\\sigma&\sqrt{1+\sigma^2}
		\end{bmatrix}\begin{bmatrix}
		\alpha & \\
		& -\alpha\\
		\end{bmatrix}\begin{bmatrix}
		\sqrt{1+\sigma^2}&\sigma\\\sigma&\sqrt{1+\sigma^2}
	\end{bmatrix}
	=\begin{bmatrix}
		\alpha & \\
		& -\alpha\\
	\end{bmatrix}
	.
\]
There are three mutually exclusive subcases:
\begin{enumerate}[(i)]
	\item both $(\what \Lambda,\what J)$ and $(\Lambda,J)$ are  block pairs of type \block{c}{$1$}. We have
		\begin{align*}
			&\inf_{\Sigma\in\group D^+\atop U_{\pm},V_{\pm}\in \group P^u}\trace(\what \Lambda  X^{\HH}\Lambda X)	\\
			&=\inf _{\sigma \ge 0 \atop \theta, \what{\theta} \in[0,2 \pi)}\trace\left(\begin{bsmallmatrix}
					\what \alpha& -\ii\what\beta\ee^{\ii\what\theta}\\
					\ii\what\beta\ee^{-\ii\what\theta}& -\what\alpha\\
					\end{bsmallmatrix}\begin{bsmallmatrix}
					\sqrt{1+\sigma^2}&\sigma\\\sigma&\sqrt{1+\sigma^2}
					\end{bsmallmatrix}\begin{bsmallmatrix}
					\alpha& -\ii\beta\ee^{\ii\theta}\\
					\ii\beta\ee^{-\ii\theta}& -\alpha\\
					\end{bsmallmatrix}\begin{bsmallmatrix}
					\sqrt{1+\sigma^2}&\sigma\\\sigma&\sqrt{1+\sigma^2}
			\end{bsmallmatrix}\right)
			\\& =\inf _{\sigma \ge 0 \atop \theta, \what{\theta} \in[0,2 \pi)} \beta \what\beta \left[(1+\sigma^2 )(\ee ^{\ii (\theta-\what{\theta})}+\ee ^{\ii (\what{\theta}-\theta)})-\sigma^2 (\ee ^{\ii (\theta+\what{\theta})}+\ee ^{-\ii (\what{\theta}+\theta)})\right]+2\alpha\what\alpha
			\\ &=\inf _{\sigma \ge 0 \atop \theta, \what{\theta} \in[0,2 \pi)} 2 \beta \what\beta \left[(1+\sigma^2 ) \cos (\theta-\what{\theta})-\sigma^2  \cos (\theta+\what{\theta})\right]+2\alpha\what\alpha
			\\ &=\inf _{\sigma \ge 0 \atop \theta, \what{\theta} \in[0,2 \pi)}  2 \beta \what\beta \left[\cos (\theta-\what{\theta})+2 \sigma^2  \sin \theta \sin \what{\theta}\right]+2\alpha\what\alpha
			\\ & =-\infty;
		\end{align*}
	\item
		$(\what \Lambda,\what J)$ is a  block pair of type \block{c}{$1$}
		and $(\Lambda,J)$ consists of two  pairs of type  \block{r}{$1$}.
		We have
		\begin{align*}
			&\inf_{\Sigma\in\group D^+\atop U_{\pm},V_{\pm}\in \group P^u}\trace(\what \Lambda  X^{\HH}\Lambda X)	\\
			&=\inf _{\sigma \ge 0 \atop \what{\theta} \in[0,2 \pi)}\trace\left(\begin{bsmallmatrix}
					\what\alpha& -\ii\what\beta\ee^{\ii\what\theta}\\
					\ii\what\beta\ee^{-\ii\what\theta}& -\what\alpha\\
					\end{bsmallmatrix}\begin{bsmallmatrix}
					\sqrt{1+\sigma^2}&\sigma\\\sigma&\sqrt{1+\sigma^2}
					\end{bsmallmatrix}\begin{bsmallmatrix}
					\lambda_+& \\
					& -\lambda_-\\
					\end{bsmallmatrix}\begin{bsmallmatrix}
					\sqrt{1+\sigma^2}&\sigma\\\sigma&\sqrt{1+\sigma^2}
			\end{bsmallmatrix}\right)
			\\&=\inf _{\sigma \ge 0 \atop \what{\theta} \in[0,2 \pi)} (\lambda_+ -\lambda_- ) \what\beta  \ii (\ee ^{-\ii  \what{\theta}}-\ee ^{\ii  \what{\theta}}) \sigma \sqrt{1+\sigma^2 }+\what\alpha(\lambda_++\lambda_-)
			\\&=\inf _{\sigma \ge 0 \atop  \what{\theta} \in[0,2 \pi)}  2(\lambda_+ -\lambda_- ) \what\beta  \sigma \sqrt{1+\sigma^2 } \sin \what{\theta}+\what\alpha(\lambda_++\lambda_-)
			\\& =-\infty,
		\end{align*}
		because $\lambda_+\ne \lambda_-$; otherwise $\Lambda=\lambda_+ J$ which has been excluded from subcase (3);
	\item
		$(\Lambda,J)$ is a  block pair of type \block{c}{$1$} and $(\what \Lambda,\what J)$
		consists of two  pairs of type  \block{r}{$1$}. This
		is similar to subcase (ii) we just considered, with the same conclusion: the infimum is $-\infty$.
\end{enumerate}

Consider, in general, $n=\what n>2$ and at least one  block pair of type \block{c}{$1$} is contained in $(\Lambda,J)$
or $(\what \Lambda,\what J)$ or both. Suppose for the moment that
$(\what \Lambda,\what J)$ contains a block pair of type \block{c}{$1$}
with $\what\alpha,\,\what\beta\in\bbR$ and $\what\beta>0$.
%
%
With the same reasoning we have employed in \cref{eq:J-unitary:half:inf},
picking $\Sigma=\sigma \be_1\be_1^{\T}$ and suitable permutation matrices $V_{\pm},U_{\pm}$ of apt sizes,
we get
\begin{align*}
	\MoveEqLeft[0]
	\inf_{\Sigma\in\group D^+\atop U_{\pm},V_{\pm}\in \group P^u}\trace(\what \Lambda  X^{\HH}\Lambda X)
	\\
	&=\inf_{\Sigma\in\group D^+\atop U_{\pm},V_{\pm}\in \group P^u}\trace\left(
		\begin{bsmallmatrix}
			V_+\what \Lambda_+^cV_+^{\HH} & -\ii V_+\what\Omega^c V_-^{\HH}\\
			\ii V_-\what\Omega^c V_+^{\HH}& -V_-\what \Lambda_-V_-^{\HH}\\
		\end{bsmallmatrix}
		\begin{bsmallmatrix}
			(I + \wtd\Sigma\wtd\Sigma^{\HH} )^{1/2} &  \wtd\Sigma \\
			\wtd\Sigma  & (I + \wtd\Sigma^{\HH}\wtd\Sigma )^{1/2}\\
	\end{bsmallmatrix} \times\right. \\
	&\hspace{3cm}\left.\begin{bsmallmatrix}
			U_+^{\HH}\Lambda_+^cU_+ & -\ii U_+^{\HH}\Omega^c U_-\\ \ii U_-^{\HH}\Omega^{c H} U_+& -U_-^{\HH}\Lambda_-^cU_-\\
			\end{bsmallmatrix}\begin{bsmallmatrix}
			(I + \wtd\Sigma\wtd\Sigma^{\HH} )^{1/2} &  \wtd\Sigma \\
			\wtd\Sigma  & (I + \wtd\Sigma^{\HH}\wtd\Sigma )^{1/2}\\
	\end{bsmallmatrix}  \right) \\
	&\le \inf_{\sigma>0}\trace\left(\begin{bsmallmatrix}
			\what \alpha & &-\ii\what \beta\ee^{\ii\what\theta}\\ & *\\\ii\what \beta\ee^{-\ii\what\theta} & & -\what\alpha \\ &&&*\\
			\end{bsmallmatrix}\begin{bsmallmatrix}
			(1 + \sigma^2 )^{1/2} & &  \sigma & \\
			& I & & 0\\
			\sigma &  & (1 + \sigma^2 )^{1/2} &\\
			& 0 && I\\
			\end{bsmallmatrix}\begin{bsmallmatrix}
			+ & &+\\ & *\\ +& & + \\ &&&*\\
			\end{bsmallmatrix}\begin{bsmallmatrix}
			(1 + \sigma^2 )^{1/2} & &  \sigma & \\
			& I & & 0\\
			\sigma &  & (1 + \sigma^2 )^{1/2} &\\
			& 0 && I\\
	\end{bsmallmatrix}\right)
	\\&\hspace{10em}\left(\text{where $\begin{bsmallmatrix}
				+ & +\\+& +
				\end{bsmallmatrix}$ is either $\begin{bsmallmatrix}
				\alpha &-\ii \beta\ee^{\ii\theta}\\\ii \beta\ee^{-\ii\theta} & -\alpha
				\end{bsmallmatrix}$ or $\begin{bsmallmatrix}
				\lambda_i & \\ & -\lambda_j
	\end{bsmallmatrix}$}\right)
	\\
	&= -\infty.
\end{align*}
In summary, the infimum is $-\infty$ as long as block pairs of type \block{c}{$1$} are involved, while if only \block{r}{$1$} block pairs are involved, it is turned into the case already considered in \cref{sec:the-simple-case}.

\subsection{Involving block pairs of types \block{r}{$1$},\block{c}{$1$}, and \block{$\infty$}{$1$}}
\label{ssec:BP:case2}
It suffices to consider the case that at least one pair of type \block{$\infty$}{$1$} is contained
in $(\Lambda,J)$ because the case of involving block pairs of types \block{r}{$1$} and \block{c}{$1$}
has already been dealt with in \cref{ssec:BP:case1} and our discussion prior to the subsection
excludes any possibility that $(\what\Lambda,\what J)$ may contain any block pair of type \block{$\infty$}{$1$}.
We write
\begin{equation}\label{eq:BP:case2:pf-0}
	\Lambda=\begin{bmatrix}
		\Lambda_r & \\ & \Lambda_{\infty}
	\end{bmatrix},\quad
	J=\begin{bmatrix}
		J_r & \\ & 0
	\end{bmatrix}, \quad
	X=\begin{bmatrix}
		X_r \\ X_{\infty}
	\end{bmatrix},
\end{equation}
where $\Lambda_{\infty}$ is diagonal with diagonal entries $\pm1$, to get
\begin{align}
	\inf_{X^{\HH}JX=\what J}\trace(\what \Lambda  X^{\HH}\Lambda X)
	&=\inf_{X_r^{\HH}J_rX_r=\what J}\trace\big(\what \Lambda\big[X_r^{\HH}\Lambda_r X_r
	+X_{\infty}^{\HH}\Lambda_{\infty} X_{\infty}\big]\big) \notag\\
	&=\inf_{X_r^{\HH}J_rX_r=\what J}\trace(\what \Lambda  X_r^{\HH}\Lambda_r X_r)
	+\inf_{X_{\infty}}\trace(\what \Lambda  X_{\infty}^{\HH}\Lambda_{\infty} X_{\infty})	.
	\label{eq:BP:case2:pf-1}
\end{align}

Consider the second term in \cref{eq:BP:case2:pf-1}, which is an infimum over $X_{\infty}$ without any constraint.
Without loss of generality, we may assume that $\what\Lambda$ is real diagonal; otherwise, since $\what\Lambda$ is Hermitian,
we let $\what\Lambda=Q\wtd\Lambda Q^{\HH}$ where $Q$ is an orthogonal matrix and $\wtd\Lambda$ is diagonal, and we get
\[
	\trace(\what \Lambda  X_{\infty}^{\HH}\Lambda_{\infty} X_{\infty})
	=\trace(Q\wtd\Lambda Q^{\HH} X_{\infty}^{\HH}\Lambda_{\infty} X_{\infty})
	=\trace(\wtd\Lambda (X_{\infty}Q)^{\HH}\Lambda_{\infty} (X_{\infty}Q)).
\]
Let $\what\Lambda=\diag(\what\lambda_1,\ldots,\what\lambda_{\what n})$ and
$\Lambda_{\infty}=\diag(\lambda_{\infty,1},\ldots,\lambda_{\infty,t})$ where $t\ge 1$. We have
\[
	\inf_{X_{\infty}}\trace(\what \Lambda  X_{\infty}^{\HH}\Lambda_{\infty} X_{\infty})
	=\inf_{X_{\infty}} \sum_{i,j} \what \lambda_{j} \lambda_{\infty,i}\abs{x_{\infty;ij}}^2,
\]
where we have written $X_{\infty}=[x_{\infty;ij}]$. Since $X_{\infty}$ is arbitrary,  each $\abs{x_{\infty;ij}}^2\ge 0$
can be made $0$ or arbitrarily large. Hence
\[
	\inf_{X_{\infty}}\trace(\what \Lambda  X_{\infty}^{\HH}\Lambda_{\infty} X_{\infty})>-\infty
	\quad\text{if and only if all $\what \lambda_{j} \lambda_{\infty,i}\ge 0$},
\]
in which case, the infimum is $0$. Notice that $\lambda_{\infty,i}=\pm 1$. There are three possible situations
for all $\what \lambda_{j} \lambda_{\infty,i}\ge 0$:
\begin{enumerate}
	\item all $\what \lambda_{j}=0$ if both $\pm 1$ appear among all $\lambda_{\infty,i}$;
	\item all $\what \lambda_{j}\ge 0$ if all $\lambda_{\infty,i}=1$;
	\item all $\what \lambda_{j}\le 0$ if all $\lambda_{\infty,i}=-1$.
\end{enumerate}
The first situation is not allowed because it implies $\what\Lambda=0$ and hence $\what A=0$ which is excluded to begin with.
Therefore, we conclude
\begin{equation}\label{eq:BP:case2:pf-2}
	\inf_{X_{\infty}}\trace(\what \Lambda  X_{\infty}^{\HH}\Lambda_{\infty} X_{\infty})>-\infty
	\,\,\text{if and only if either $\what\Lambda\succeq 0$, $\Lambda_{\infty}=I$
	or $\what\Lambda\preceq 0$, $\Lambda_{\infty}=-I$}.
\end{equation}

Consider now the first term in \cref{eq:BP:case2:pf-1}, which falls into the case in \cref{ssec:BP:case1}.
In light of \cref{eq:BP:case2:pf-2}, to see when
\begin{equation}\label{eq:BP:case2:pf-3}
	\inf_{X_r^{\HH}J_rX_r=\what J}\trace(\what \Lambda  X_r^{\HH}\Lambda_r X_r)>-\infty
\end{equation}
and what the infimum is,
it suffices to investigate what will happen when either
$\what\Lambda\succeq 0$, $\Lambda_{\infty}=I$
or $\what\Lambda\preceq 0$, $\Lambda_{\infty}=-I$. We have the following:
\begin{enumerate}
	\item Suppose $\what\Lambda\succeq 0$, $\Lambda_{\infty}=I$. Then $(\what\Lambda,\what J)\succeq 0$ and, by the result of \cref{ssec:BP:case1},
		\cref{eq:BP:case2:pf-3} holds if and only if
		$(\Lambda_r,J_r)\succeq 0$ and $(J_r,\what\Lambda,\what J)$ is proper, which is the same as
		that $(\Lambda,J)\succeq 0$ and $(J,\what\Lambda,\what J)$ is proper because of
		\cref{eq:BP:case2:pf-0} and $\Lambda_{\infty}=I$.
	\item Suppose $\what\Lambda\preceq 0$, $\Lambda_{\infty}=-I$. Then $(\what\Lambda,\what J)\preceq 0$ and, by the result of \cref{ssec:BP:case1},
		\cref{eq:BP:case2:pf-3} holds if and only if
		$(\Lambda_r,J_r)\preceq 0$ and $(-J_r,-\what\Lambda,-\what J)$ is proper, which is the same as
		that $(\Lambda,J)\preceq 0$ and $(-J,-\what\Lambda,-\what J)$ is proper because of
		\cref{eq:BP:case2:pf-0} and $\Lambda_{\infty}=-I$.
\end{enumerate}

\subsection{Involving block pairs of types \block{r}{$p$} with $p\le 2$, \block{c}{$1$}, and \block{$\infty$}{$1$}}
\label{ssec:BP:case3}
It suffices to consider there are some block pairs of type \block{r}{$2$} in the mix; otherwise the situation
has already been taken care of in \cref{ssec:BP:case2}. Let $\varepsilon>0$ be arbitrary tiny, and perturb
each block pair $(\eta K_2(\alpha), \eta F_2)$ of type \block{r}{$2$} according to
\begin{equation}\label{eq:BP:case3:pf-0}
	K_2(\alpha) \to K_2(\alpha)+\varepsilon \be_1  \be_1^{\T},
\end{equation}
which turns the block pair$(\eta K_2(\alpha), \eta F_2)$ to two block pairs \block{r}{$1$} with eigenvalues
$\alpha\pm\sqrt{\varepsilon}$, respectively, and both are continuous in  $\varepsilon$
and go to $\alpha$ as $\varepsilon\to 0^+$. As a result, both $\Lambda$ and $\what\Lambda$ are possibly perturbed to
$\Lambda_{\varepsilon}$ and $\what\Lambda_{\varepsilon}$, satisfying
\begin{equation*}
	\Lambda_{\varepsilon}
	\begin{cases}
		\equiv\Lambda, \quad&\text{if no block pair of type \block{r}{$2$} in $(\Lambda, J)$}, \\
		\to\Lambda, \quad&\text{as $\varepsilon\to 0^+$}.
	\end{cases}
\end{equation*}
The same holds true for $(\what \Lambda_{\varepsilon},\what J)$.
Consider now $(\Lambda_{\varepsilon},J)$ and $(\what\Lambda_{\varepsilon},\what J)$ in which
only block pairs of types \block{r}{$1$}, \block{c}{$1$}, and \block{$\infty$}{$1$}
are possibly involved. It is important to note that both $J$ and $\what J$ are not perturbed, leaving
$\inertia_{\pm}(J)$ and $\inertia_{\pm}(\what J)$ unaffected.
Note that, for any $\alpha\in\bbR$, $\big(K_2(\alpha), F_2\big)\succeq 0$ but $\big(K_2(\alpha), F_2\big)\not\preceq 0$.

\begin{lemma}\label{lm:case3-pf-1}
	Given $\varepsilon>0$,
	$\big(K_2(\alpha)+\varepsilon \be_1  \be_1^{\T}\big)-\lambda_0 F_2\succeq 0$ if and only if
	$\alpha-\sqrt{\varepsilon}\le\lambda_0\le\alpha+\sqrt{\varepsilon}$.
\end{lemma}

\begin{proof}
	Notice that
	\[
		K_2(\alpha)+\varepsilon \be_1  \be_1^{\T}-\lambda_0 F_2
		=\begin{bmatrix}
			\varepsilon & \alpha-\lambda_0 \\
			\alpha-\lambda_0 & 1
		\end{bmatrix}.
	\]
	Since $\varepsilon>0$, the matrix is positive semidefinite if and only if its determinant $\varepsilon-(\alpha-\lambda_0)^2\ge 0$.
\end{proof}

The next lemma is stated in terms of $(\Lambda,J)$. It is clearly valid if $(\Lambda,J)$ is replaced with $(\what\Lambda,\what J)$.

\begin{lemma}\label{lm:case3-pf-2}
	Suppose that $(\Lambda,J)$ is a direct sum of block pairs of types
	\block{r}{$p$} with $p\le 2$, \block{c}{$1$}, and \block{$\infty$}{$1$} and that
	each block pair of type \block{r}{$2$} is perturbed according to \cref{eq:BP:case3:pf-0} where
	$\varepsilon>0$.
	\begin{enumerate}[{\rm (a)}]
		\item If there is a positive sequence $\set{\varepsilon_i}_{i=1}^{\infty}$ converging to $0$, i.e.,
			$0<\varepsilon_i\to 0$ as $i\to \infty$, such that
			$(\Lambda_{\varepsilon_i},J)\succeq 0$ for all $i$, then
			$(\Lambda,J)\succeq 0$, and $(\Lambda,J)$ can only contain block pairs $\big(\eta K_2(\alpha), \eta F_2\big)$ with $\eta=1$ and the same $\alpha$
			for all block pairs of type \block{r}{$2$}, in which case
			$\Lambda-\lambda_0 J\succeq 0$ with $\lambda_0=\alpha$ and only with $\lambda_0=\alpha$.

			Conversely, if $(\Lambda,J)\succeq 0$, then $(\Lambda_{\varepsilon},J)\succeq 0$.
		\item If there is a positive sequence $\set{\varepsilon_i}_{i=1}^{\infty}$ converging to $0$, i.e.,
			$0<\varepsilon_i\to 0$ as $i\to \infty$, such that
			$(\Lambda_{\varepsilon_i},J)\preceq 0$ for all $i$, then
			$(\Lambda,J)\preceq 0$, and $(\Lambda,J)$ can only contain block pairs $\big(\eta K_2(\alpha), \eta F_2\big)$ with $\eta=-1$ and the same $\alpha$
			for all block pairs of type \block{r}{$2$}, in which case
			$\Lambda-\lambda_0 J\preceq 0$ with $\lambda_0=\alpha$ and only with $\lambda_0=\alpha$.

			Conversely, if $(\Lambda,J)\preceq 0$, then $(\Lambda_{\varepsilon},J)\preceq 0$.
	\end{enumerate}
\end{lemma}

\begin{proof}
	We will only prove item (a). The same argument with minor modifications can be used to prove item (b).

	Suppose that $(\Lambda_{\varepsilon_i},J)\succeq 0$ for all $i$, which means that for each $i$ there exists $\mu_i$
	such that $\Lambda_{\varepsilon_i}-\mu_i J\succeq 0$. By \cite[Lemma 3.8]{lilb:2013}, $\abs{\mu_i}$ can be taken no bigger than
	the absolute values of the finite eigenvalues of $(\Lambda_{\varepsilon_i},J)$.
	Under the perturbation, the finite eigenvalues of matrix pairs $(\Lambda_{\varepsilon_i},J)$ are uniformly bounded because
	they converges to the finite eigenvalues of $(\Lambda,J)$. Hence $\set{\mu_i}_{i=1}^{\infty}$ is bounded and thus has a
	convergent subsequence $\set{\mu_i}_{i\in\bbI}$, say converging to $\lambda_0$, where $\bbI$ is an infinite subset of $\set{1,2,\ldots,}$.
	Letting $\bbI\ni i\to\infty$ in $\Lambda_{\varepsilon_i}-\mu_i J\succeq 0$ yields
	$\Lambda-\lambda_0 J\succeq 0$.

	If $(\Lambda,J)$ ever contains a block pair $\big(\eta K_2(\alpha), \eta F_2\big)$, then we will have
	\[
		\eta(K_2(\alpha)+\varepsilon_i \be_1  \be_1^{\T}-\mu_i F_2)
		=\eta\begin{bmatrix}
			\varepsilon_i & \alpha-\mu_i \\
			\alpha-\mu_i & 1
		\end{bmatrix}\succeq 0
	\]
	for all $i$, which implies $\eta=1$ and $\alpha-\sqrt{\varepsilon_i}\le\mu_i\le\alpha+\sqrt{\varepsilon_i}$.
	Letting $i\to\infty$ yields $\mu_i\to\alpha$. If $(\Lambda,J)$ also contains another
	block pair $\big(\wtd \eta K_2(\wtd \alpha), \wtd \eta F_2\big)$ of the same type. Using the same argument as we just did,
	we find $\wtd\eta=1$ and also $\mu_i\to\wtd\alpha$ yielding $\wtd\alpha=\alpha$.

	Conversely, suppose that $(\Lambda,J)\succeq 0$. If no block pair of type \block{r}{$2$} is involved in $(\Lambda,J)$, then
	$\Lambda_{\varepsilon}\equiv \Lambda$ and hence no proof is necessary. If $(\Lambda,J)$
	does contain block pairs of \block{r}{$2$}, then these block pairs must be
	$\big( K_2(\alpha),  F_2\big)$ with the same $\alpha$. Therefore the only $\lambda_0$ that makes $\Lambda-\lambda_0 J\succeq 0$
	is $\lambda_0=\alpha$ which also makes $\big(K_2(\alpha)+\varepsilon \be_1  \be_1^{\T}\big)-\lambda_0 F_2\succeq 0$ for
	any $\varepsilon>0$. By the way how $\Lambda$ is perturbed to  $\Lambda_{\varepsilon}$, we find
	$\Lambda_{\varepsilon}-\lambda_0 J\succeq 0$.
\end{proof}

By the results of \cref{ssec:BP:case1,ssec:BP:case2}, we conclude that
\begin{equation}\label{eq:BP:case3:pf-2}
	\inf_{ X^{\HH}J X=\what J}\trace(\what \Lambda_{\varepsilon} X^{\HH}\Lambda_{\varepsilon} X)>-\infty
\end{equation}
if and only if
one of the following two scenarios occurs:
\begin{enumerate}[(1)]
	\item both $(\Lambda_{\varepsilon},J)$ and $(\what\Lambda_{\varepsilon},\what J)$ are positive  semidefinite pairs and
		$(J,\what\Lambda_{\varepsilon},\what J)$ is proper;
	\item both $(\Lambda_{\varepsilon},J)$ and $(\what\Lambda_{\varepsilon},\what J)$ are negative semidefinite pairs and
		$(-J,-\what\Lambda_{\varepsilon},-\what J)$ is proper.
\end{enumerate}
Let $\set{\varepsilon_i}_{i=1}^{\infty}$ be a positive sequence that converges to $0$, i.e.,
$0<\varepsilon_i\to 0$ as $i\to \infty$. Since there are only two scenarios here, there is a subsequence
$\set{\varepsilon_i}_{i\in\bbI}$ such that one of the two scenarios holds true for all $i\in\bbI$. In the case when
\[
	\text{
		for all $i\in\bbI$, both $(\Lambda_{\varepsilon_i},J)\succeq 0$, $(\what\Lambda_{\varepsilon_i},\what J)\succeq 0$, and
		$(J,\what\Lambda_{\varepsilon_i},\what J)$ is proper,
	}
\]
we have
both $(\Lambda,J)\succeq 0$, $(\what\Lambda,\what J)\succeq 0$, and
$(J,\what\Lambda,\what J)$ is proper, as a consequence of  \Cref{lm:case3-pf-2}. Similarly, we can conclude that if
\[
	\text{
		for all $i\in\bbI$, both $(\Lambda_{\varepsilon_i},J)\preceq 0$, $(\what\Lambda_{\varepsilon_i},\what J)\preceq 0$, and
		$(-J,-\what\Lambda_{\varepsilon_i},-\what J)$ is proper,
	}
\]
then
both $(\Lambda,J)\preceq 0$, $(\what\Lambda,\what J)\preceq 0$, and
$(-J,-\what\Lambda,-\what J)$ is proper.

With either scenario, the infimum in \cref{eq:BP:case3:pf-2} has a close formula as in \cref{eq:main-inf}, or it applied to
$(-\Lambda_{\varepsilon},-J)$ and $(-\what\Lambda_{\varepsilon},-\what J)$. Because of the continuity of
these eigenvalues with respect to $\varepsilon$, the limit of the infimum exists as $\varepsilon\to 0^+$.
Since the perturbation does not affect $\inertia_{\pm}(J)$ and $\inertia_{\pm}(\what J)$ at all, the limit
takes the same form as \cref{eq:main-inf}, or it applied to
$(-\Lambda,-J)$ and $(-\what\Lambda,-\what J)$.

\subsection{Involving block pairs of all possible types in \cref{eq:BP-possible}}
\label{ssec:BP:case4}
In this subsection, we will allow  all block pairs of types in \cref{eq:BP-possible}
to possibly appear in $(\Lambda,J)$ and $(\what\Lambda,\what J)$.
Block pairs of types  in 
\begin{equation}
	\label{tbl:last-ones}
	\text{
		\begin{tabular}{|c|c|c|c|c|}
			\hline
			type & \block{s}{$2 p+1$} & \block{c}{$p$} & \block{r}{$p$} & \block{$\infty$}{$p$} \\ \hline
			$p$ & $p\ge 1$ & $p\ge 2$ & $p\ge 3$ & $p\ge 2$ \\
			\hline
		\end{tabular}
	}
\end{equation}
remain to be included for considerations, as we have already considered
\block{r}{$p$} with $p\le 2$, \block{c}{$1$}, and \block{$\infty$}{$1$},

Notice that a positive/negative semidefinite matrix pair does not contain
any block pair of these types in \Cref{tbl:last-ones} in its canonical form (see \Cref{rk:LambdaJ-PD}).
In what follows, we will show that
\begin{equation}\label{eq:BP:case4:pf-2}
	\inf_{X^{\HH}JX=\what J}\trace(\what \Lambda  X^{\HH}\Lambda X)=-\infty
\end{equation}
if any block pair of these types in \Cref{tbl:last-ones} is contained in either $(\Lambda,J)$ or
$(\what\Lambda,\what J)$ or both, besides \block{r}{$p$} with $p\le 2$, \block{c}{$1$}, and \block{$\infty$}{$1$}.
The idea is to perturb $(\Lambda,J)$ and/or $(\what\Lambda,\what J)$ to
$(\Lambda_{\varepsilon},J)$ and/or $(\what\Lambda_{\varepsilon},\what J)$
such that
\begin{enumerate}
	\item $\Lambda_{\varepsilon}\to \Lambda$ and $\what\Lambda_{\varepsilon}\to\what\Lambda$ as $\varepsilon\to 0$,
	\item for sufficiently tiny $\varepsilon>0$, the canonical forms of $(\Lambda_{\varepsilon},J)$ and $(\what\Lambda_{\varepsilon},\what J)$ contain block pairs of types
		\block{r}{$1$},\block{c}{$1$}, and \block{$\infty$}{$1$} only, and that has been investigated in \cref{ssec:BP:case2}, and
		either
		\begin{equation}
			\inf_{X^{\HH}JX=\what J}\trace(\what \Lambda_{\varepsilon}  X^{\HH}\Lambda_{\varepsilon} X)
			=-\infty, \label{eq:BP:case4:pf-3-1}
		\end{equation}
		or
		\[
			\lim_{\varepsilon\to 0}
			\inf_{X^{\HH}JX=\what J}\trace(\what \Lambda_{\varepsilon}  X^{\HH}\Lambda_{\varepsilon} X)
			=-\infty. 
		\]
\end{enumerate}
Hence, we justify our claim \cref{eq:BP:case4:pf-2} for the case of interest.

Specifically, we perturb the {\em first\/} block elements in block pairs of the types in \Cref{tbl:last-ones}
as follows:
\begin{align*}
	K_{2p+1}(0)=\begin{bmatrix}
		&&K_p(0)\\ &0&\be_1^{\T}\\ K_p(0)&\be_1&\\
	\end{bmatrix}
	&\rightarrow
	\begin{bmatrix}
		&&K_p(\ii\varepsilon)\\&\varepsilon&\be_1^{\T}\\ K_p(-\ii\varepsilon)&\be_1&\\
	\end{bmatrix},
	\\
	\begin{bmatrix}
		0 & K_p (\alpha+\ii  \beta) \\
		K_p (\alpha+\ii  \beta) &
	\end{bmatrix}
	&\rightarrow\begin{bmatrix}
		0 & K_p (\alpha+\ii  \beta)+\varepsilon \be_1  \be_1^{\T}  \\
		K_p (\alpha+\ii  \beta)+\varepsilon \be_1  \be_1^{\T}
	\end{bmatrix},
	\\
	F_p  & \rightarrow F_p +\varepsilon \be_1  \be_1^{\T}, 
	\\
	K_p (\alpha)  & \rightarrow  K_p (\alpha)+\varepsilon \be_1  \be_1^{\T}.  
\end{align*}
We restrict $\varepsilon>0$, except for \block{$\infty$}{$2$}, for which $\varepsilon<0$ is also allowed.
\begin{enumerate}[(i)]
	\item A block pair of type \block{c}{$p$} with $p\ge 2$ and eigenvalues $\alpha \pm \ii  \beta$
		generates $p$ block pairs of type \block{c}{$1$} with eigenvalues
		\[
			\alpha \pm \ii  \beta \pm \varepsilon^{1 / p} \exp (\ii \frac{2 \pi   j}{p}), j=0, \ldots, p-1.
		\]
		Among them there are genuine conjugate complex eigenvalues.
	\item A block pair of type \block{s}{$2 p+1$} with $p\ge 1$ generates a block pair of type \block{$\infty$}{$1$}
		and a block pair of type \block{c}{$p$} with a pair of genuine conjugate complex eigenvalues,
		and eventually generates a block pair of type \block{$\infty$}{$1$} and $2p$ block pairs of type \block{c}{$1$} with eigenvalues some of which are genuine conjugate complex eigenvalues.
	\item A block pair of type \block{r}{$p$} with $p\ge 3$ and eigenvalues $\alpha$ generates $p$ block pairs of type \block{c}{$1$}
		or \block{r}{$1$} with eigenvalues
		\[
			\alpha+\varepsilon^{1 / p} \exp (\ii  \frac{2 \pi   j}{p}), j=0, \ldots, p-1.
		\]
		Among them there are genuine conjugate complex eigenvalues.
	\item A block pair of type \block{$\infty$}{$p$} with $p\ge 2$ generates a block pair of type \block{$\infty$}{$1$} and $p-1$
		block pairs of type \block{c}{$1$} or \block{r}{$1$} with eigenvalues
		\[
			\varepsilon^{-1/(p-1)}\exp(\ii\frac{2\pi j}{p-1}),j=0, \dots, p-2.
		\]
		Among them there are genuine conjugate complex eigenvalues if $p\ge 3$.
\end{enumerate}
After perturbations, $(\Lambda_{\varepsilon},J)$ and $(\what\Lambda_{\varepsilon},\what J)$ themselves are no longer
in their canonical forms as the ones specified in \Cref{lm:HCF}. But they can be turned into
their canonical forms, in which only possible block pairs of types \block{c}{$1$}, \block{r}{$1$}, and \block{$\infty$}{$1$} show up.
When any one of (i), (ii), (iii), and (iv) with $p>2$ occurs, we will have at least one block pair of
type \block{c}{$1$} in the canonical form, and hence \cref{eq:BP:case4:pf-3-1} holds
by the results in \cref{ssec:BP:case2}, which implies \cref{eq:BP:case4:pf-2}.
%
%

It remains to consider (iv) with $p=2$ and only block pairs of type \block{$\infty$}{$2$},
besides \block{r}{$p$} with $p\le 2$ and \block{$\infty$}{$1$}, can show up. We exclude any
block pair of type \block{c}{$1$} because if such a block pair exists, we will have, after perturbations,
\cref{eq:BP:case4:pf-3-1} and hence \cref{eq:BP:case4:pf-2}.
Note that
block pair of type \block{$\infty$}{$2$} can only be contained
in $(\Lambda,J)$ according to \Cref{eq:BP-possible}, while $(\what\Lambda,\what J)$ contains
possibly block pairs of type \block{r}{$p$} with $p\le 2$.
Without needing to perturb any block pair of type \block{r}{$2$} in $(\what\Lambda,\what J)$, if any,  \Cref{lm:Tinfty2} below shows that \Cref{eq:BP:case4:pf-2} holds.

\begin{lemma}\label{lm:Tinfty2}
	If $(\Lambda, J)$ contains a block pair of type \block{$\infty$}{$2$}, then
	\Cref{eq:BP:case4:pf-2} holds.
\end{lemma}

\begin{proof}
	We perturb any block pair of type \block{$\infty$}{$2$} in $(\Lambda, J)$ as
	\[
		\eta ( F_2,  K_2(0))
		\to\eta \left(\begin{bmatrix}
				\varepsilon & 1 \\
				1 & 0
				\end{bmatrix},\begin{bmatrix}
				0 & 0 \\
				0 & 1
			\end{bmatrix}
		\right)
		\sim\eta \left(\begin{bmatrix}
				\varepsilon & 0 \\
				0 & -\frac 1{\varepsilon}
				\end{bmatrix},\begin{bmatrix}
				0 & 0 \\
				0 & 1
			\end{bmatrix}
		\right)
		\sim \left(\begin{bmatrix}
				\sign(\eta\varepsilon) & 0 \\
				0 & -\frac 1{\varepsilon}
				\end{bmatrix},\begin{bmatrix}
				0 & 0 \\
				0 & \eta
			\end{bmatrix}
		\right),
	\]
	where ``$\sim$'' stands for ``is congruent to''. Without loss of generality, we may assume
	that $\eta ( F_2,  K_2(0))$ is the last block pair in $(\Lambda, J)$.
	As a result,
	\begin{align*}
		(\Lambda, J)&=\left(\begin{bmatrix}
				\Lambda_r & 0 \\
				0 & \eta  F_2
				\end{bmatrix},\begin{bmatrix}
				J_r & 0 \\
				0 & \eta K_2(0)
			\end{bmatrix}
		\right) \\
		\to \,(\Lambda_{\varepsilon}, J)&=\left(\begin{bmatrix}
				\Lambda_r & 0 \\
				0 & \eta ( F_2+\varepsilon \be_1\be_1^{\T})
				\end{bmatrix},\begin{bmatrix}
				J_r & 0 \\
				0 & \eta K_2(0)
			\end{bmatrix}
		\right) \\
		\sim\,(\underline{\Lambda}_{\varepsilon},\underline{J})&=\left(\begin{bmatrix}
				\Lambda_r & 0 & 0 \\
				0 & -\frac 1{\varepsilon} & 0 \\
				0 & 0 & \sign(\eta\varepsilon)
				\end{bmatrix},\begin{bmatrix}
				J_r & 0 & 0 \\
				0 & \eta & 0 \\
				0 & 0 & 0
		\end{bmatrix}\right) \\
		&=:\left(\begin{bmatrix}
				\underline{\Lambda}_{\varepsilon;r} & 0  \\
				0  & \sign(\eta\varepsilon)
				\end{bmatrix},\begin{bmatrix}
				\underline{J}_r & 0 \\
				0 & 0
		\end{bmatrix}\right),
	\end{align*}
	the canonical form of $(\Lambda_{\varepsilon}, J)$.
	Similarly to \cref{eq:BP:case2:pf-1}, we have
	\begin{align}
		\inf_{X^{\HH}JX=\what J}\trace(\what \Lambda  X^{\HH}\Lambda_{\varepsilon} X)
		&=\inf_{X^{\HH}\underline{J}X=\what J}
		\trace(\what\Lambda  X^{\HH}\underline{\what\Lambda}_{\varepsilon} X) \notag \\
		&=\inf_{X_r^{\HH}\underline{J}_rX_r=\what J}
		\trace(\what\Lambda  X_r^{\HH}\underline{\Lambda}_{\varepsilon;r} X_r)
		+\inf_{X_{\infty}}
		\trace(\what\Lambda  X_{\infty}^{\HH}\sign(\eta\varepsilon) X_{\infty})	.
		\label{eq:BP:case4:pf-5}
	\end{align}
	As in our argument after \cref{eq:BP:case2:pf-1},  we will consider the last infimum in \Cref{eq:BP:case4:pf-5}.
	For that purpose we may assume $\what\Lambda$ is diagonal. Because the freedom in making
	either $\varepsilon>0$ or $\varepsilon<0$, we can show the infimum over $X_{\infty}$
	is $-\infty$, unless $\what\Lambda=0$ which is excluded in \Cref{thm:J-unitary:total}.
	Hence we have \Cref{eq:BP:case4:pf-3-1} and hence \Cref{eq:BP:case4:pf-2}.
\end{proof}

Summarizing what we have done so far leads to the main result in \Cref{thm:J-unitary:total}.

\begin{remark}\label{rk:definite}
	So far, we have been assumed that $B$ is genuinely indefinite. We now comment on the proof for the case when $B$ is positive or negative semidefinite. It suffices to consider the case $B\succeq 0$, because when $B\preceq 0$, we can consider the infimum of interest for $(-A,-B)$
	and $(-\what A,-\what B)$, instead. Suppose that $B\succeq 0$. Then
	$\what B\succ 0$ because $\what B$ is
	always nonsingular and $\inertia_+(\what B)\le \inertia_+(B)$ and $\inertia_-(\what B)\le \inertia_-(B)$ by
	\cref{eq:inertia(BtB)}. We again transform matrix pairs $(A,B)$ and $(\what A,\what B)$ to their canonical forms
	as in \Cref{lm:HCF}. We will still have \cref{eq:simplified'} but with fewer possible types of  block pairs
	to consider in $(\Lambda, J)$ and $(\what \Lambda,\what J)$ than those in \cref{eq:BP-possible}. Specifically,
	\begin{equation*}
		(\Lambda, J):\,\text{\block{$\infty$}{$p$} with $p\le 2$, \block{r}{$1$}}; \quad
		(\what \Lambda,\what J): \,\text{\block{r}{$1$}}.
	\end{equation*}
	Also $\what J=I_{\what n}$ always. If no block pair of type \block{$\infty$}{$2$} shows up in $(\Lambda, J)$, then
	it falls into a special situation of \cref{ssec:BP:case2} where, though under the scope of $B$
	being genuinely indefinite, no argument there relies on that. If, however,
	\block{$\infty$}{$2$} is involved, we can use
	\Cref{lm:Tinfty2}.
\end{remark}

\section{Concluding remarks}\label{sec:conclusion}
We have established a trace minimization principle for two Hermitian matrix pairs $(A,B)$ and $(\what A,\what B)$:
\begin{equation}\label{eq:problem'}
	\inf_{\what BX^{\HH}BX=I_{\what n}}\trace(\what A X^{\HH}AX),
\end{equation}
where $A,\,B\in\bbC^{n\times n}$ and $\what A,\,\what B\in\bbC^{\what n\times\what n}$ are all Hermitian, and $\what n\le n$.
It is the most general one up to date, encompassing Fan's trace minimization principle \cite{fan:1949}
(for $\what A=\what B=I_{\what n}$ and $B=I_n$) and its straightforward extension (for $\what A=\what B=I_{\what n}$ and
positive definite $B$),
%
and most recent ones \cite{kove:1995,lilb:2013,liwz:2023} reviewed at \cref{sec:introduction}.
In those recent investigations,  the notion of positive semidefinite matrix pair was introduced:
a Hermitian matrix pair $(A,B)$ is positive (negative) semidefinite if there exists $\lambda_0\in\bbR$ such that
$A-\lambda_0B$ is positive (negative) semidefinite.

For investigating \cref{eq:problem'}, we introduced yet another notion for a Hermitian matrix triplet
$(B,\what A,\what B)$ being \emph{proper} in \Cref{dfn:proper}. We showed that
the infimum in \cref{eq:problem'} is finite is and only if either both $(A, B)$ and $(\what A, \what B)$ are positive  semidefinite pairs and
$(B,\what A,\what B)$ is proper, or
both  are negative semidefinite pairs and
$(-B,-\what A,-\what B)$ is proper, assuming $\what A\ne 0$, $A\ne \mu B$ for any $\mu\in \bbR$,
and $\what A\ne\what \mu\what B$ for any $\what\mu\in \bbR$ when $n=\what n$. A close formula for the infimum is given in terms of the finite eigenvalues of
the two semidefinite matrix pairs.
%
%

In \cite[Example 3.1]{liwz:2023}, the following example (in the notation here):
\begin{gather*}
	\mu=2,\,\,A=\begin{bmatrix}
		1 & \\ & \mu\\
	\end{bmatrix},\,\,
	B=\begin{bmatrix}
		1 & \\ & -1\\
	\end{bmatrix},\,\, \what B=B,\\
	\sigma=\frac {\sqrt{18-6\sqrt 2}}6,\,\,
	\Omega=\begin{bmatrix}
		1 & \\ & 1/4
	\end{bmatrix}, \,\,
	Q=\begin{bmatrix}
		\sqrt{1-\sigma^2} & -\sigma\\
		\sigma & \sqrt{1-\sigma^2}
	\end{bmatrix},\,\,
	\what A=Q^{\HH}\Omega Q,
\end{gather*}
was given to demonstrate that
the infimum in \cref{eq:problem'} may not be any sum of the products between the eigenvalues of $\what A$ and some of
the ones of $(A,B)$, as a justification for an assumption of \cite[Theorem~3.2]{liwz:2023}.
This now can be well explained by our \Cref{thm:J-unitary:total} in this paper, i.e.,
it is the eigenvalues of $(\what A,\what B)$, not $\what A$ alone, that should appear in
the infimum.
For the example, the eigenvalues of $(\what A,\what B)$ and of $(A,B)$ are
\[
	\what\lambda_1^+=\frac 12\sqrt 2,\,\,\what\lambda_1^-=-\frac 14\sqrt 2,
	\quad\text{and}\quad
	\lambda_1^+=1,\,\,\lambda_1^-=-2,
\]
respectively. Hence the infimum, by \Cref{thm:J-unitary:total}, is
$
\what\lambda_1^+\lambda_1^++\what\lambda_1^-\lambda_1^-=\sqrt 2.
$

{\small
	\bibliographystyle{plain}
	\bibliography{tracemin3-min}
}

\end{document}